\documentclass{amsart}

\input xy
\xyoption{all}

\usepackage{fullpage}
\usepackage[colorlinks, linkcolor=blue, citecolor=blue, urlcolor=red!80!gray]{hyperref} % include "pagebackref" in the brackets if you want the bibliography to show where the citations appear in the paper
\usepackage[square,numbers]{natbib}
\usepackage{amsfonts}
\usepackage[all]{xy}
\usepackage{amsmath, amssymb, faktor}
\usepackage{mathtools}
\usepackage{graphicx}
\usepackage{tikz}
\usepackage{tikz-cd} % for commutative diagram environment
\usepackage{faktor} % for slanted quotients
\usepackage{enumitem} % for enumerating with letters
\usepackage[skip=10pt, indent=15pt]{parskip}
\usepackage[normalem]{ulem} % for stikethrough "\sout{}" % [normalem] makes it that so \emph is not changed to \underline

\newtheorem{same}{This should never appear}[section]
\newtheorem{definition}[same]{Definition}

\newtheorem{remark}[same]{Remark}
\newtheorem{theorem}[same]{Theorem}

\newtheorem{lemma}[same]{Lemma}
\newtheorem{fact}[same]{Fact}
\newtheorem{question}[same]{Question}
\newtheorem{corollary}[same]{Corollary}
\newtheorem{prop}[same]{Proposition}

\newtheorem{definition*}{Definition}
\newtheorem*{theorem*}{Theorem}
\newtheorem*{theorem-1}{Theorem \ref{thm: WO <-> Superstable}}
\newtheorem*{corollary-2}{Corollary \ref{cor: old GCH stability}}
\newtheorem*{theorem-3}{Theorem \ref{thm: Baer-like for p.i.}}

\newcommand{\id}{\operatorname{id}}
\newcommand{\K}{\mathbf{K}}

\newcommand{\CM}{\operatorname{C}_M^C}
\newcommand{\SAct}{{S\mbox{-}\mathbf{Act}}}
\newcommand{\KActp}{\K_{\SAct_p}}
\newcommand{\LS}{\operatorname{LS}}
\newcommand{\gtp}{\mathbf{gtp}}
\newcommand{\gS}{\mathbf{gS}}
\newcommand{\pp}{\mathbf{pp}}

\newcommand{\Stab}{\operatorname{Stab}}
\newcommand{\card}[1]{| #1 |}%{\operatorname{card}( #1 )}
\newcommand{\cardS}{\operatorname{card}(S)}

    \newbox\noforkbox \newdimen\forklinewidth
    \setbox0\hbox{$\textstyle\smile$}
    \setbox1\hbox to \wd0{\hfil\vrule width \forklinewidth depth-2pt
     height 10pt \hfil}
    \setbox\noforkbox\hbox{\lower 2pt\box1\lower 2pt\box0\relax}
    \def\unionstick{\mathop{\copy\noforkbox}\limits}
    
    \newcommand{\nf}{\unionstick}
    \newcommand{\dnfb}[4]{#2 \overset{#4}{\underset{#1}{\overline{\nf}}} #3}
    \newcommand{\dnf}[4]{#2 \overset{#4}{\underset{#1}{\nf}} #3}

    \newbox\noforkboxp
    \setbox\noforkboxp\hbox{%
      \copy\noforkbox
      \hbox to 0pt{\hss\raise 4pt\hbox{\scalebox{0.5}{p}}\hskip 1pt}%adds a 0 width box as to not affect \overline later, \hss shifts all the way to the right, then \raise and \hskip are used to move vertically and horizontally as needed (can use negative values)
    }

    \def\unionstickp{\mathop{\vcenter{\hbox{\copy\noforkboxp}}}\limits}

    \newcommand{\nfp}{\unionstickp}
    \newcommand{\dnfbp}[4]{#2 \overset{#4}{\underset{#1}{\overline{\nfp}}} #3}
    \newcommand{\dnfp}[4]{#2 \overset{#4}{\underset{#1}{\nfp}} #3}

\title[Acts with Pure Embeddings]{On the Abstract Elementary Class of Acts with Pure Embeddings}

\author[Feigert, Herden, Mazari-Armida]{Jonathan Feigert, Daniel Herden, Marcos Mazari-Armida}
    
\date{\today}

\address{
	Department of Mathematics,
	Baylor University,
	Sid Richardson Building,
	1410 S.~4th Street,
	Waco, TX 76706, USA}

\email{\texttt{
    Jonathan\_Feigert1@baylor.edu,
    Daniel\_Herden@baylor.edu,
    Marcos\_Mazari@baylor.edu
    }}

\thanks{The second author was supported by Simons Foundation grant MPS-TSM-00007788.
	The third author was partially supported by NSF grant DMS-2348881 and Simons Foundation grant MPS-TSM-00007597.}

\begin{document}


\begin{abstract}
    We study the abstract elementary class of acts with pure embeddings.
    In particular, we show that stability and superstability in this class can be characterized in terms of the monoid $S$ being LO (for every $s,t \in S$, we have that $s \in St$ or $t \in Ss$) and weakly noetherian (every ideal is finitely generated), respectively.
    Moreover, under mild set-theoretic assumptions, we characterize the stability spectrum via the minimal cardinality of a generating set for every ideal.
    As an application, we obtain a Baer-like criterion for pure injective acts when $S$ is LO.
    We use this to provide an alternative proof that the class of acts has enough pure injectives when $S$ is LO.
\end{abstract}

\maketitle
\tableofcontents
%-----%-----%-----%
%-----%-----%-----%
%-----%-----%-----%
%-----%-----%-----%
\section{Introduction} \label{sec: intro}
%-----%-----%-----%
%-----%-----%-----%
%-----%-----%-----%
%-----%-----%-----%

Abstract elementary classes (AECs for short) are a model-theoretic framework introduced by Shelah \cite{shelah1987nonelementaryclasses} to capture the semantic structure of non-first-order theories.
Historically, the study of AECs has been focused on developing the abstract theory, but more recently there has been mounting evidence of the meaningful connections to algebra.
These connections have been primarily studied in the context of modules (see \cite{boney2026module, mazari2021flatrings, mazari2021model, mazari2025limit, MazariRosickyRelativeInjective}).
The aim of this paper is to continue to develop the study of these interactions but for the class of acts.

$S$-acts (also known as polygons) can be seen both as ``non-additive modules'' and as a generalization of $G$-sets where $S$ is only assumed to be a monoid rather than a group.
A robust introduction to acts can be found in \cite{kilp2011monoids}.
The class of $S$-acts with a partial order given by embeddings was studied as an AEC in \cite{mazariarmida2025abstractelementaryclassacts}.
In this paper, we again study the class of $S$-acts but now with a partial order given by pure embeddings.
Pure embeddings lie in between embeddings and elementary embeddings, and can be defined in various ways: model-theoretically in terms of preserving certain first-order formulas, algebraically in terms of preserving exact sequences, and category-theoretically as directed colimits of split monomorphisms.

Pure injective objects --- objects that are injective with respect to pure embeddings --- have been widely studied.
For modules, they form a class between cotorsion and injective modules. 
Moreover, every module is a pure submodule of a pure injective module, namely its double character module (see \cite{enochs2011homological}).
Early results about pure injectives in the context of universal algebra (where pure injectivity is known as equational compactness) can be found in \cite{wenzel1970subdirect}, \cite{banaschewski1972equational}, \cite{banaschewski1974equational}, and \cite{normak1980purity}.

The focus of this paper, acts with pure embeddings, differs both from acts with embeddings and  from modules with pure embeddings in the following ways.
It is known that both of these other AECs are always stable \cite{mazariarmida2025abstractelementaryclassacts, Kucera2020Modules}, one of the key model-theoretic dividing lines. 
However, in our setting, stability is equivalent to the monoid $S$ being locally linearly preordered (LO for short), see Theorem \ref{thm: LO <-> Stable}.
We borrow this terminology from Mustafin \cite{mustafin1988stability} where it was shown to be equivalent to the first-order theory of every $S$-act being stable, and from Cox et al. \cite{cox2025cofibrantgenerationpuremonomorphisms} where it was shown to be equivalent to pure embeddings being cofibrantly generated.
Moreover, the AEC of $S$-acts with pure embeddings in general fails to have joint embedding, a property that holds for both acts with embeddings and modules with pure embeddings.
Surprisingly, it appears that joint embedding may generally be an interesting algebraic property, as it is in our setting (see Proposition \ref{prop: JEP <-> local zeros}).
Superstability, another of the key model-theoretic dividing lines, has been found to also be an algebraically significant property in many contexts.
For $S$-acts with embeddings, superstability is equivalent to $S$ being weakly noetherian \cite{mazariarmida2025abstractelementaryclassacts}.
For $R$-modules with pure embeddings, it is equivalent to $R$ being pure-semisimple \cite{mazari2021superstability}.
In our setting, superstability characterizes monoids that are both weakly noetherian and LO (see Theorem \ref{thm: WO <-> Superstable}).
Furthermore, under some mild set-theoretic assumptions, we completely characterize the stability spectrum of acts with pure embeddings (see Theorem \ref{thm: GCH stability}).
Such a characterization has been obtained for acts with embeddings \cite{mazariarmida2025abstractelementaryclassacts}, but is not currently known for modules with pure embeddings.

In the recent paper 
\cite{cox2025cofibrantgenerationpuremonomorphisms}, pure embeddings were studied in the slightly more general setup of a presheaf category $\textbf{Set}^\mathcal{C}$ of set-valued functors for a small category $\mathcal{C}$ (this reduces to the class of $S$-acts when $\mathcal{C}$ has only one object, i.e., $\mathcal{C}$ is a monoid $S$).
However, the primary result there was that pure embeddings are cofibrantly generated if and only if $\mathcal{C}$ is LO.
In this paper, we largely assume that monoids are LO and focus on finer, model-theoretic notions related to stability.
These are studied using Galois types rather than the category-theoretic independence relations.
This requires a careful consideration of properties such as joint embedding and local character, which did not play a significant role in \cite{cox2025cofibrantgenerationpuremonomorphisms}.
Furthermore, as an application to acts theory, we obtain a Baer-like criterion for pure injective acts when $S$ is LO.

% \addtocontents{toc}{\setcounter{tocdepth}{-10}}
\subsection*{Summary of Paper}
In Section \ref{sec: prelims}, we provide some necessary background information on acts, AECs, and independence relations which serve as the main pieces of machinery throughout much of the paper.
Additionally, we prove some general results about the AEC of $S$-acts with pure embeddings ($\KActp$) such as the equivalence between equality of Galois types and the equality of pp-types (see Theorem \ref{thm: gtp = pp}).
This result will be a key ingredient in many of our other proofs, in part because it quickly implies tameness.
Although similar results have been obtained for some classes of modules, there are certain classes (such as flat modules) where it is not currently known.
The approach employed in this paper is novel and may be able to be adapted to work in other settings.
We also characterize the monoids for which joint embedding holds (see Proposition \ref{prop: JEP <-> local zeros}).
When applied to the case when $S$ is a group, this shows that joint embedding and categoricity only occur when $S$ is trivial (see Corollary \ref{cor: JEP for G-sets}).

In Section \ref{sec: main results}, we first prove that stability is equivalent to $S$ being LO (see Theorem \ref{thm: LO <-> Stable}).
Furthermore, we show that superstability is again algebraically interesting by using it to characterize monoids that are both LO and weakly noetherian.
    \begin{theorem-1}
        $\KActp$ is superstable if and only if $S$ is weakly noetherian and LO.
    \end{theorem-1}
In Subsection \ref{sub: main - stability spectrum}, we characterize the stability cardinals under some mild set-theoretic assumptions via a cardinal $\gamma(S)$ that measures the minimal cardinality of a generating set for every ideal of $S$ (see Theorem~\ref{thm: GCH stability}). In particular, we note the corresponding result under the assumption of the Generalized Continuum Hypothesis (GCH).
    \begin{corollary-2}
        Assume GCH. Let $S$ be an LO monoid and $\lambda \ge (\cardS+\aleph_0)^+$.
        Then, $\KActp$ is stable in $\lambda$ if and only if $\lambda^{<\gamma(S)} = \lambda$.
    \end{corollary-2}
In Section \ref{sec: limit models}, we prove a Baer-like criterion for pure injectivity.
    \begin{theorem-3}
        Let $S$ be an LO monoid.
        An $S$-act $A$ is pure injective if and only if, for any $S$-acts $B \leq_p C$ such that $\card{C} \leq \cardS + \aleph_0$, every $S$-homomorphism $f : B \rightarrow A$ can be extended to an $S$-homomorphism from $C$ to $A$.
    \end{theorem-3}
Using limit models, this is then used to provide a more direct proof that every $S$-act is a pure subact of a pure injective act when $S$ is LO, recovering a result from \cite{cox2025cofibrantgenerationpuremonomorphisms}.
We further study limit models in Subsection~\ref{sub: limit models - spectrum} by focusing on those that can be jointly embedded with the zero act.
Under some mild-set theoretic assumptions, we characterize the spectrum of these limit models (see Theorem \ref{thm: limit model iso}).

\subsection*{Acknowledgments}
We would like to thank Will Boney, Ji\v{r}\'{\i} Rosick\'{y}, and  Michael D. Walton for discussions around the topics of this paper.

% \changelocaltocdepth{2} % ToC shows subsections after this point
%-----%-----%-----%
%-----%-----%-----%
%-----%-----%-----%
%-----%-----%-----%
\section{Preliminaries and Basic Results} \label{sec: prelims}
%-----%-----%-----%
%-----%-----%-----%
%-----%-----%-----%
%-----%-----%-----%

    We begin by introducing some notation and terminology from both acts theory and model theory that will be used throughout this paper.
    For a more in-depth introduction, we refer the reader to \cite{kilp2011monoids} for acts, \cite{baldwin2009categoricity} for abstract elementary classes, and \cite{Lieberman_2019} for independence relations.

%-----%-----%-----%
%-----%-----%-----%
%-----%-----%-----%
%-----%-----%-----%
\subsection{Acts} \label{sub: prelims - acts}
%-----%-----%-----%
%-----%-----%-----%
%-----%-----%-----%
%-----%-----%-----%

    Let $S$ be a monoid with identity element $1 \in S$.
    We denote the cardinality of $S$ by $\cardS$.
    A \emph{(left) $S$-act} is a non-empty set $A$ together with a multiplication $S \times A \rightarrow A$ such that $1a = a$ and $s(ta) = (st)a$ for every $s,t \in S$ and $a \in A$.
    A function $f : A \rightarrow B$ between $S$-acts is an \emph{S-homomorphism} if $f(sa) = sf(a)$ for every $s \in S$ and $a \in A$.
    In this case, we say that $f$ is an \emph{embedding} (resp. an \emph{$S$-isomorphism}) if $f$ is also injective (resp. bijective).
    If the inclusion map from $A$ to $B$ is an $S$-homomorphism, we say that $A$ is a \emph{subact} of $B$ and write $A \leq B$.
    For $X \subseteq A$, the \emph{subact generated by $X$} is $SX = \{sx \mid s \in S, ~ x \in X\}$.
    When $X = \{a\}$, we will simply write $Sa$. For $S$-acts $M \leq A,B$, we say that $A$ and $B$ are \emph{isomorphic over $M$} if there exists an $S$-isomorphism $f : A \rightarrow B$ with $f\upharpoonright M = \id_M$.
    
    An \emph{$S$-congruence} on an $S$-act $A$ is an equivalence relation $\rho$ on $A$ such that $(sa,sb) \in \rho$ whenever $(a,b) \in \rho$ and $s\in S$.
    In this case, we denote the $\rho$-equivalence class containing $a \in A$ by $[a]_\rho$.
    When $\rho$ is clear from context, we will simply write $[a]$.
    For any subset $X \subseteq A \times A$, the \emph{$S$-congruence generated by $X$}, denoted $\rho_X$, is the smallest $S$-congruence on $A$ containing $X$ and is characterized as follows: $(a,b) \in \rho_X$ if and only if $a = b$ or there exists some $0 < n < \omega$, and some $s_i \in S$ and $p_i,q_i \in A$ with $(p_i,q_i)$ or $(q_i,p_i) \in X$ for $1 \leq i \leq n$ such that $a = s_1p_1$, $s_iq_i = s_{i+1}p_{i+1}$ for $1 \leq i \leq n-1$, and $s_nq_n = b$ (see \cite[Corollary 1.1.7]{kilp2011monoids} for the analogous characterization of equivalence relations on a set $A$ generated by a subset $X \subseteq A \times A$).
    
    We say that $I \subseteq S$ is a \emph{left ideal} of $S$ if $SI \subseteq I$, i.e., if $I$ is empty or a subact of $S$.
    We denote by $\gamma(S)$ the minimum infinite cardinal such that every left ideal of $S$ is generated by fewer than $\gamma(S)$ elements, and define $\gamma_r(S)$ to be the minimum regular cardinal that is greater than or equal to $\gamma(S)$, i.e.,
    \[
        \gamma_r(S) = \begin{cases}
            \gamma(S) & \text{if} ~ \gamma(S) ~ \text{is regular}, \\
            \gamma(S)^+ & \text{if} ~ \gamma(S) ~ \text{is singular},
        \end{cases} \tag{2.1} \label{eq:gamma_r}
    \]
    where $\gamma(S)^+$ denotes the \emph{successor} of $\gamma(S)$.
    If $\gamma(S) = \gamma_r(S) = \aleph_0$, then we say that $S$ is \emph{weakly (left) noetherian}.
    
    We denote the category of $S$-acts with $S$-homomorphisms by $\SAct$.
    Two constructions that we will make use of repeatedly are the coproduct and the pushout.
    In $\SAct$, the \emph{coproduct} of $S$-acts $A$ and $B$ is the disjoint union and will be denoted by $A \coprod B$.
    The \emph{pushout} of a pair of $S$-homomorphisms $B_1 \xleftarrow{f_1} A \xrightarrow{f_2} B_2$ is given by the $S$-act $P = (B_1 \coprod B_2)/\rho$ together with the $S$-homomorphisms $B_1 \xrightarrow{g_1} P \xleftarrow{g_2} B_2$ where $\rho$ is the $S$-congruence on $B_1 \coprod B_2$ generated by $\{(f_1(a),f_2(a)) \mid a \in A\}$ and $g_i$ is the canonical inclusion of $B_i$ into $B_1\coprod B_2$ composed with the canonical projection for $i = 1,2$.
    For a pair of $S$-homomorphisms $B_1 \xrightarrow{h_1} C \xleftarrow{h_2} B_2$ with $h_1 \circ f_1 = h_2 \circ f_2$, the unique $S$-homomorphism $k : P \rightarrow C$ such that $h_i = k \circ g_i$ for $i = 1,2$ is given by $k([b]_\rho) = h_i(b)$ if $b \in B_i$ (see \cite[Proposition 2.2.16]{kilp2011monoids}).
    In this case, we may simply write that $B_1 \xrightarrow{g_1} P \xleftarrow{g_2} B_2$ is the pushout of $B_1 \xleftarrow{f_1} A \xrightarrow{f_2} B_2$, and we call $k : P \rightarrow C$ the \emph{pushout-induced map} of $B_1 \xrightarrow{h_1} C \xleftarrow{h_2} B_2$.
    Observe that when $f_1,f_2$ are inclusion maps and $B_1 \cap B_2 = A$, then we may take the pushout to be $B_1 \cup B_2$ with $g_1,g_2$ as inclusion maps and $h_1 \cup h_2$ as the pushout-induced map.
    
    Although it will not be used as much as the coproduct and the pushout, we recall for completeness the category-theoretic definition of the pullback, the dual notion of the pushout.
    The \emph{pullback} of a pair of $S$-homomorphisms $B_1 \xrightarrow{g_1} C \xleftarrow{g_2} B_2$ is 
    an $S$-act $M$ together with a pair of $S$-homomorphisms $B_1 \xleftarrow{f_1} M \xrightarrow{f_2} B_2$ such that $g_1 \circ f_1 = g_2 \circ f_2$ and, for every pair of $S$-homomorphisms $B_1 \xleftarrow{h_1} A \xrightarrow{h_2} B_2$ such that $g_1 \circ h_1 = g_2 \circ h_2$, there exists a unique $S$-homomorphism $k: A \rightarrow M$ such that $h_i = f_i \circ k$ for $i = 1,2$. In this case, the commutative square of $S$-homomorphisms $(f_1, f_2, g_1, g_2)$ will be called a \emph{pullback square}.
    See \cite[Proposition~2.2.5]{kilp2011monoids} for a description of pullbacks in $\SAct$.
    In particular, when $g_1$, $g_2$ are inclusion maps and $B_1 \cap B_2 \neq \varnothing$, we may take the pullback to be $B_1 \cap B_2$ with $f_1,f_2$ as inclusion maps. 
    
    Additionally, we will use the following definition from \cite{mustafin1988stability}.

    \begin{definition}[{\cite{mustafin1988stability}}]
        Let $S$ be a monoid, $A$ be an $S$-act, $X \subseteq A$, and $a,b \in A\setminus X$.
        We say $a$ and $b$ are \emph{connected outside $X$} if there exist $n < \omega$ and $a_0, \dots, a_n \in A \setminus X$ such that $a_0 = a$, $a_n = b$, and $a_i \in Sa_{i+1}$ or $a_{i+1} \in Sa_i$ for all $i < n$.
        In this case, we call the $(n+1)$-tuple $(a = a_0, a_1, \dots, a_n = b)$ a \emph{connecting path} from $a$ to $b$ in $A$ outside of $X$.
        Let
        \[
            \operatorname{C}^A_X(a) = \{b \in A \mid a \text{ and } b \text{ are connected outside } X\},
        \]
        and for $Y \subseteq A\setminus X$
        \[
            \operatorname{C}^A_X(Y) = \bigcup_{a \in Y}\operatorname{C}^A_X(a).
        \]
    \end{definition}
    
    It is clear that, for fixed $X \subseteq A$, being connected outside $X$ defines an equivalence relation on $A \setminus X$ where $\operatorname{C}^A_X(a)$ is exactly the equivalence class containing $a$.

    In this paper, we will be particularly interested in $S$-homomorphisms that are pure embeddings (a stronger notion than an embedding).
    We consider equivalent logical and algebraic definitions of pure embeddings.
    For a fixed language, a \emph{positive primitive formula} or \emph{pp-formula} is a first-order formula that is built from atomic formulas using only conjunctions and existential quantifiers.
    Up to logical equivalence, every pp-formula $\varphi(x_1,\dots,x_n)$ can be written as
    \[
        \exists x_{n+1} \dots x_{n+m} \psi(x_1,\dots,x_{n+m})
    \]
    for some $\psi(x_1,\dots,x_{n+m})$ that is a conjunction of atomic formulas. In the language of $S$-acts (which consists only of unary function symbols, one for each element of $S$), $\psi(x_1,\dots,x_{n+m})$ is simply a finite system of equations of the form $sx_i = tx_j$ for some $s,t \in S$ and $1 \leq i,j \leq n+m$.
    Moreover, if $A$ is an $S$-act and $a_1,\dots,a_n \in A$, then $\psi(a_1,\dots,a_n,x_{n+1},\dots,x_{n+m})$ is simply a finite system of equations with constants from~$A$.
    Thus, one can view the semantics of pp-formulas in terms of the solvability of finite systems of equations of the forms $sx = ty$ and $sx = a$ where $x,y$ are variables, $a\in A$ is a constant, and $s,t \in S$.
    
    We will commonly abbreviate $\varphi(x_1,\dots,x_{n})$ as $\varphi(\overline{x})$, where $\overline{x}$ denotes the vector of free variables of $\varphi$ and $\ell(\overline{x})$ the length of this vector.
    It immediately follows from the definitions that if $f : A \rightarrow B$ is an $S$-homomorphism, $\varphi(\overline{x})$ is a pp-formula, and $\overline{a} \in A^{\ell(\overline{x})}$, then $A \models \varphi(\overline{a})$ implies that $B \models \varphi(f(\overline{a}))$.
    For a pure embedding, the reverse implication also holds.
    That is to say, an $S$-homomorphism $f : A \rightarrow B$ is a \emph{pure embedding} if, for every pp-formula $\varphi(\overline{x})$ and every $\overline{a} \in A^{\ell(\overline{x})}$, we have
    \[
    A \models \varphi(\overline{a}) ~ \text{if and only if} ~ B \models \varphi(f(\overline{a})).
    \]
    It can be easily seen that pure embeddings are in fact embeddings since $x_1 = x_2$ is a pp-formula.
    If the inclusion map from $A$ to $B$ is a pure embedding, we say that $A$ is a \emph{pure subact} of $B$ and write $A \leq_p B$.

    Equivalently, we may define pure embeddings in terms of systems of equations.
    This will be the perspective primarily taken in this paper.
    Let $f : A \rightarrow B$ be an embedding.
    For every system of equations $\Sigma$ with constants from $f(A)$, denote by $f^{-1}(\Sigma)$ the associated system of equations formed from $\Sigma$ by replacing each constant with its unique preimage in $A$.

    \begin{fact}
        Let $f : A \rightarrow B$ be an embedding.
        Then the following are equivalent.
        \begin{enumerate}
            \item[$(1)$] $f : A \rightarrow B$ is a pure embedding.

            \item[$(2)$] For every finite system of equations $\Sigma$ with constants from $f(A)$, we have that $f^{-1}(\Sigma)$ has a solution in $A$ if and only if $\Sigma$ has a solution in $B$.
        \end{enumerate}
    \end{fact}
    
    Pure embeddings interact particularly nicely with coproducts and pushouts in $\SAct$.
    
    \begin{fact}\label{fact: preserving purity}\
        \begin{enumerate}
            \item[$(1)$] Let $f_i : A \rightarrow B_i$ be pure embeddings for $i = 1,2$.
            Then $f_i(A) \leq_p B_1 \coprod B_2$ for $i = 1,2$.

            \item[$(2)$] Let $g_i : A_i \rightarrow B_i$ be pure embeddings for $i = 1,2$.
            Then $g_1(A_1) \coprod g_2(A_2) \leq_p B_1 \coprod B_2$.

            \item[$(3)$] Let $B_1 \xrightarrow{g_1} P \xleftarrow{g_2} B_2$ be the pushout of $B_1 \xleftarrow{f_1} A \xrightarrow{f_2} B_2$.
            If $f_i$ is an embedding (resp. pure embedding), then $g_{3-i}$ is an embedding (resp. pure embedding) for $i = 1,2$.
        \end{enumerate}
    \end{fact}

        \begin{proof}
            We provide a model-theoretic proof\footnote{
                A proof using category theory can be found in \cite[Proposition 15]{adamek2004pure}.
            }
            of (3) in the cases when $f_2$ is an embedding or a pure embedding.
            Similar arguments can be used to prove (1) and (2).
            The presented method of manipulating systems of equations originates from \cite[Lemma 4.4]{cox2025cofibrantgenerationpuremonomorphisms}, and further instances of this novel technique can be found in the proofs Theorem \ref{thm: gtp = pp} and Lemma \ref{lem: pure closure rel M}.

            Without loss of generality, we may assume that $B_1 \cap B_2 = \varnothing$, so that $P = (B_1 \cup B_2)/\rho$ where $\rho$ is generated by $\{(f_1(a),f_2(a)) \mid a \in A\}$.
            We first characterize the $\rho$-equivalence classes.
            If $(b,b') \in \rho$ with $b \neq b'$, then there exists some $0 < n < \omega$ and some $s_i \in S$, $j_i \in \{1,2\}$, and $a_i \in A$ for $1 \leq i \leq n$ such that the following system of equations holds: 
            \[
                b = s_1f_{j_1}(a_1), \quad s_if_{3-j_i}(a_i) = s_{i+1}f_{j_{i+1}}(a_{i+1}) ~\text{for}~ 1 \leq  i \leq n - 1, \quad \text{and} \quad s_nf_{3-j_n}(a_n) = b'.
            \]
            Without loss of generality, we may choose $s_i = 1$ for all $i$ because $f_1,f_2$ are $S$-homomorphisms.
            Moreover, since $B_1 \cap B_2 = \varnothing$, observe that $b \in B_{j_1}$, $b' \in B_{3-j_n}$, and $3-j_i = j_{i+1}$ for $1 \leq i \leq n-1$.
            Thus, the system of equations becomes: 
            \[
                b = f_{k_0}(a_1), \quad f_{k_i}(a_i) = f_{k_i}(a_{i+1}) ~\text{for}~ 1 \leq  i \leq n - 1, \quad \text{and} \quad f_{k_n}(a_n) = b',
            \]
            where $k_i = k_0$ if $i$ is even and $k_i = 3-k_0$ if $i$ is odd.
            Since $f_2$ is an embedding, we can always reduce the system whenever an equation of the form $f_2(a_i) = f_2(a_{i+1})$ appears and, ultimately, to one of the following:
            \begin{enumerate}[label = (\roman*)]
                \item $b = f_1(a) = b'$ for some $a \in A$, if $b,b' \in B_1$;

                \item $b = f_2(a)$, $f_1(a) = f_1(a')$, and $f_2(a') = b'$ for some $a,a' \in A$, if $b,b' \in B_2$;

                \item $b = f_2(a)$ and $f_1(a) = b'$ for some $a \in A$, if $b \in B_2$ and $b' \in B_1$; or

                \item $b = f_1(a)$ and $f_2(a) = b'$ for some $a \in A$, if $b \in B_1$ and $b' \in B_2$.
            \end{enumerate}
            Since $b \neq b'$, we know that case (i) is impossible.

            Observe that $g_1$ is an embedding because $g_1(b) = [b]_\rho$ for every  $b \in B_1$ and each $\rho$-equivalence class contains at most one element of $B_1$ by the argument above.
            
            Now suppose that $f_2$ is a pure embedding, and let $\Sigma$ be a finite system of equations in variables $x_1, \dots, x_n, y_1, \dots, y_m$ with constants from $g_1(B_1)$ and solution $g_1(b_1),\dots,g_1(b_n)$, $g_2(b'_1),\dots,g_2(b'_m) \in P$ where each $b_i \in B_1$ and each $b'_i \in B_2 \setminus f_2(A)$.
            From $\Sigma$, we build two related systems of equations $\Delta_1$ and $\Delta_2$ in variables $x_1, \dots, x_n$ and $y_1,\dots,y_m$, respectively, as follows.
            \begin{itemize}
                \item If $(sx_i = tx_j) \in \Sigma$, then $g_1(sb_i) = g_1(tb_j)$.
                Since $g_1$ is an embedding, we know that $sb_i = tb_j$.
                In this case, let $(sx_i = tx_j) \in \Delta_1$.

                \item If $(sy_i = ty_j) \in \Sigma$ and $sb'_i = tb'_j$, then let $(sy_i = ty_j) \in \Delta_2$.

                \item If $(sy_i = ty_j) \in \Sigma$ and $sb'_i \neq tb'_j$, then $g_2(sb'_i) = g_2(tb'_j)$.
                Hence Case (ii) above applies, and so $sb'_i = f_2(a)$, $f_1(a) = f_1(a')$, and $f_2(a') = tb'_j$ for some $a,a'\in A$.
                In this case, let $(sy_i = f_2(a)),(ty_j = f_2(a')) \in \Delta_2$.

                \item If $(sx_i = g_1(b)) \in \Sigma$, then $g_1(sb_i) = g_1(b)$.
                Since $g_1$ is an embedding, we know that $sb_i = b$.
                In this case, let $(sx_i = b) \in \Delta_1$.

                \item If $(sy_i = g_1(b)) \in \Sigma$, then $g_2(sb'_i) = g_1(b)$.
                Hence Case (iii) above applies, and so $sb'_i = f_2(a)$ and $f_1(a) = b$ for some $a \in A$.
                In this case, let $(sy_i = f_2(a)) \in \Delta_2$.

                \item If $(sx_i = ty_j) \in \Sigma$, then $g_1(sb_i) = g_2(tb'_j)$.
                Hence Case (iv) above applies, and so $sb_i = f_1(a)$ and $f_2(a) = tb'_j$ for some $a \in A$.
                In this case, let $(sx_i = f_1(a)) \in \Delta_1$ and $(ty_j = f_2(a)) \in \Delta_2$.
            \end{itemize}
            By construction, $b_1, \dots, b_n \in B_1$ is a solution to $\Delta_1$.
            Moreover, $\Delta_2$ has constants from $f_2(A)$ and solution $b'_1, \dots, b'_m \in B_2$.
            Because $f_2$ is a pure embedding, it follows that $f_2^{-1}(\Delta_2)$ has a solution $a_1, \dots, a_m \in A$.
            It can be easily verified by checking each case that the constructions of $\Delta_1$ and $\Delta_2$ are such that this implies that $b_1, \dots,b_n$, $f_1(a_1), \dots, f_1(a_m)$ is a solution to $g_1^{-1}(\Sigma)$.
        \end{proof}

    For a fixed monoid $S$, there is a unique (up to isomorphism) $S$-act with a single element: the \emph{zero act} $\Theta = \{\theta\}$ defined by $s\theta = \theta$ for every $s \in S$.
    $\Theta$ appears to play a vital role in understanding pure embeddings and, as such, will appear repeatedly throughout this paper.
    It can be easily seen that any $S$-homomorphism $f : \Theta \rightarrow A$ is a pure embedding.
    
%-----%-----%-----%
%-----%-----%-----%
%-----%-----%-----%
%-----%-----%-----%
\subsection{Abstract Elementary Classes} \label{sub: prelims - AECs}
%-----%-----%-----%
%-----%-----%-----%
%-----%-----%-----%
%-----%-----%-----%

    An \emph{abstract elementary class \textup{(}AEC\textup{)}}  $\K = (K,\leq_\K)$ is a class $K$ of structures in some fixed language together with a partial order $\leq_\K$ on $K$ contained in the substructure relation\footnote{
        In the context of $S$-acts, substructures are the same as subacts.
    }
    such that $\K$ $\dots$
    \begin{itemize}
        \item $\dots$is \emph{closed under isomorphisms}, i.e., $M \in K$ and $M \cong N$ implies that $N \in K$, and $M \leq_\K N$ and $f: N \cong L$ implies that $f(M) \leq_\K L$,

        \item $\dots$satisfies the \emph{coherence property}, i.e., if $M,N \leq_\K L$ and $M$ is a substructure of $N$, then $M \leq_\K N$,

        \item $\dots$satisfies the \emph{L\"owenheim-Skolem axiom}, i.e., there is some cardinal $\lambda$ such that for every $N \in K$ and every subset $X \subseteq N$, there is some $M \leq_\K N$ containing $X$ with $\card{M} \leq \card{X} + \lambda$, and

        \item $\dots$satisfies the \emph{Tarski-Vaught axioms}, i.e., if $\delta$ is a limit ordinal and $\{M_i \mid i < \delta\} \subseteq K$ is a $\leq_\K$-increasing continuous chain\footnote{
            This means that $M_j \leq_\K M_i$ for every $j < i < \delta$, and $M_i = \bigcup_{j < i} M_j$ whenever $i < \delta$ is a limit ordinal.
        }, then $\bigcup_{i < \delta} M_i \in K$ and $M_j \leq_\K \bigcup_{i < \delta}M_i$ for every $j < \delta$; moreover, $M_i \leq_\K N$ for every $i < \delta$ implies that $\bigcup_{i < \delta} M_i \leq_\K N$.
    \end{itemize}
    The least infinite cardinal $\lambda$ for which $\K$ satisfies the L\"owenheim-Skolem axiom is called the \emph{L\"owenheim-Skolem number} of $\K$ and is denoted $\LS(\K)$.
    For cardinals $\lambda \geq \LS(\K)$, we write $\K_\lambda$ to denote $\{M \in K \mid \card{M} = \lambda\}$, $\K_{\leq\lambda}$ to denote $\{M \in K \mid \card{M} \leq \lambda\}$, and $\K_{<\lambda}$ to denote $\K_{\leq\lambda}\setminus\K_\lambda$.
    A function $f : M \rightarrow N$ is called a \emph{$\K$-embedding} if $f : M \cong f(M)$ and $f(M) \leq_\K N$.

    We introduce some standard properties that an AEC $\K$ might satisfy.

    \begin{definition}\label{def: AP,JEP,NMM}
        We say that $\K$ has \emph{amalgamation} if, for every pair of $\K$-embeddings $N_1 \xleftarrow{f_1} M \xrightarrow{f_2} N_2$, there is a pair of $\K$-embeddings $N_1 \xrightarrow{g_1} L \xleftarrow{g_2} N_2$ such that $g_1 \circ f_1 = g_2 \circ f_2$.
        If, in addition, we may choose $g_1,g_2$ such that $g_1(N_1) \cap g_2(N_2) = (g_1\circ f_1)(M) ~ ( = (g_2 \circ f_2)(M))$, then we say that $\K$ has \emph{disjoint amalgamation}.
        We say that $\K$ has \emph{joint embedding} if, for every $M_1,M_2 \in K$, there are $\K$-embeddings $M_1 \xrightarrow{f_1} N \xleftarrow{f_2} M_2$.
        We say that $\K$ has \emph{no maximal models} if, for every $M \in K$, there is some $N \in K$ such that $M \leq_\K N$ and $M \subsetneq N$.
    \end{definition}

    The AEC $\K_{\SAct} = (\SAct,\leq)$ of (left) $S$-acts with the subact relation is well-investigated and we recall some of its basic properties for the convenience of the reader.

    \begin{fact}[{\cite[Lemma 2.4]{mazariarmida2025abstractelementaryclassacts}}]
        $\K_{\SAct}$ is an AEC with $\LS(\K_{\SAct}) = \cardS + \aleph_0$, disjoint amalgamation, joint embedding, and no maximal models.
    \end{fact}

    In this paper, we will almost exclusively work with the class of (left) $S$-acts with the pure subact relation and will denote this by $\KActp = (\SAct,\leq_p)$.
    It turns out that $\KActp$ is also an AEC that satisfies some, but not all, of the same properties as $\K_{\SAct}$.

    \begin{lemma}\label{lem: AP, NMM, and LS(K)}
        $\KActp$ is an AEC with $\LS(\KActp) = \cardS + \aleph_0$, disjoint amalgamation, and no maximal models.
    \end{lemma}

    \begin{proof}
        Observe that the L\"owenheim-Skolem axiom follows from the Downward L\"owenheim-Skolem Theorem for first-order model theory.
        We show that $\KActp$ satisfies the Tarski-Vaught axioms.
        The other axioms can be shown similarly.
        Let $\delta$ be a limit ordinal, $\{A_i \mid i < \delta\} \subseteq S\text{-Act}$ be a $\leq_p$-increasing continuous chain, and $A = \bigcup_{i < \delta}A_i$.
    
        For fixed $i < \delta$, let $\Sigma$ be a finite system of equations with constants from $A_i$ and a solution in $A$.
        Since the solution contains only finitely many elements, we may choose $i \leq j < \delta$ so that it is contained in $A_j$.
        But because $A_i \leq_p A_j$, we have that $\Sigma$ is solvable in $A_i$ and so $A_i \leq_p A$.
        Moreover, if $B$ is an $S$-act such that $A_i \leq_p B$ for every $i < \delta$,
        then a similar argument shows that $A \leq_p B$.
        
        Disjoint amalgamation follows from Fact \ref{fact: preserving purity}(3), and no maximal models follows from Fact \ref{fact: preserving purity}(1) with $A = B_1 = B_2$ and $f_1 = f_2 = \text{id}_A$.
    \end{proof}

    Unlike $\K_{\SAct}$ and most natural examples of AECs, $\KActp$ in general fails to have joint embedding.
    As the following proposition implies, this even occurs for monoids such as the non-negative integers under addition.
    However, we can classify the monoids $S$ for which $\KActp$ has joint embedding.

    \begin{definition}[{\cite[Definition 3.6.5]{kilp2011monoids}}]
        We say that a monoid $S$ has \emph{local zeros} if, for every finite subset $F \subseteq S$, there is some $\zeta \in S$ such that $s\zeta = \zeta$ for every $s \in F$.
    \end{definition}

    \begin{prop}\label{prop: JEP <-> local zeros}
        $\KActp$ has joint embedding if and only if $S$ has local zeros.
    \end{prop}

        \begin{proof}
            \noindent\underline{Necessity:} Suppose $\KActp$ has joint embedding.
            Then, in particular, there exists an $S$-act $A$ and some pure embeddings $f : S \rightarrow A$ and $g : \Theta \rightarrow A$.
            Without loss of generality, we may assume that $g = \id_{\Theta}$.
            Let $F \subseteq S$ be a finite subset, and define a system of equations $\Sigma = \{sx = x \mid s \in F\}$.
            Observe that $\theta \in \Theta \leq A$ is a solution to $\Sigma$.
            Thus, since $f$ is a pure embedding and $\Sigma$ has no constants, $S$ also contains a solution to $f^{-1}(\Sigma) = \Sigma$.

            \noindent\underline{Sufficiency:} Suppose $S$ has local zeros, and let $A$ be an $S$-act.
            Without loss of generality, we may assume that $A \cap \Theta = \varnothing$.
            It is clear that $\Theta \leq_p A \cup \Theta$, and so it is enough to show that $A \leq_p A \cup \Theta$.
            This is the case because, if $B$ is another $S$-act with $B \leq_p B \cup \Theta$, then $A$ and $B$ can be jointly embedded into the pushout $A \cup \Theta \xrightarrow{f_1} P \xleftarrow{f_2} B \cup \Theta$ of $A \cup \Theta \xleftarrow{\id} \Theta \xrightarrow{\id} B \cup \Theta$ as in the following diagram
            \[
            \begin{tikzcd}
                A \arrow[r, phantom, "\leq_p"] & A \cup \Theta \arrow[r, "f_1"] & P \\
                & \Theta \arrow[u, phantom, sloped, "\leq_p"] \arrow[r, phantom, "\leq_p"] & B \cup \Theta \arrow[u, "f_2"] \\
                & & B \arrow[u, phantom, sloped, "\leq_p"]
            \end{tikzcd}
            \]
            where $f_1$ and $f_2$ are pure embeddings by Fact \ref{fact: preserving purity}(3).
            
            To see that $A \leq_p A \cup \Theta$, let $\Sigma$ be a finite system of equations in variables $x,y_1,\dots,y_n$ with constants from $A$ and solution $\theta,a_1,\dots,a_n$ with each $a_i \in A$. (This also covers the more general case of a finite system of equations in variables $x_1,\ldots,x_m,y_1,\dots,y_n$ with constants from $A$ and solution $\theta,\ldots,\theta,a_1,\dots,a_n$ with $a_i \in A$ via the identification $x:= x_1=\ldots=x_m$.)
            Since $\theta \notin A$, we know there is no equation in $\Sigma$ of the form $sx = a$ or $sx = ty_i$ for any $s,t \in S$ and $a \in A$.
            Writing $\Sigma_x$ for the system of equations from $\Sigma$ that contain the variable $x$ and $\Sigma_y$ for the system of equations that contain $y_i$ for some $1 \leq i \leq n$, this implies that $\Sigma_x$ has no constants and $\Sigma_x \cap \Sigma_y = \varnothing$.
            Thus, since $a_1,\dots,a_n$ is a solution to $\Sigma_y$ by construction and $\Sigma = \Sigma_x \cup \Sigma_y$, it suffices to show that $A$ contains a solution to $\Sigma_x$.
            However, if we let $F$ be the (necessarily finite) set of all $s \in S$ such that the term $sx$ occurs in some equation in $\Sigma_x$ and $\zeta \in S$ be the local zero associated to $F$, then $\zeta a$ is a solution to $\Sigma_x$ for any $a \in A$.
        \end{proof}

    Recall that a class of structures in a fixed language is \emph{categorical} in a cardinal $\lambda$ if it contains a unique model (up to isomorphism) of cardinality $\lambda$.
    The next result quickly follows from Proposition \ref{prop: JEP <-> local zeros}.

    \begin{corollary}\label{cor: JEP for G-sets}
        Let $S$ be a group.
        The following are equivalent.
        \begin{enumerate}
            \item[$(1)$] $S$ is trivial.

            \item[$(2)$] $\KActp$ has joint embedding.

            \item[$(3)$] $\SAct$ is categorical in some $\lambda \geq \cardS+\aleph_0$.

            \item[$(4)$] $\SAct$ is categorical in all $\lambda \geq \cardS + \aleph_0$.
        \end{enumerate}
    \end{corollary}
    % Proof: (1)->(2), (1)->(3), (1)->(4), (4)->(3) are clear
    % (2)->(1) follows from Prop
    % (3)->(2) uses DLS, NMM, and AP broken down in cases by cardiniality

        \begin{proof}
            The equivalence between (1) and (2) follows immediately from Proposition \ref{prop: JEP <-> local zeros}.
            It is clear that (1) implies (4), and (4) implies (3).
            Moreover,
            (3) implies (1) as $\coprod_\lambda S$ and $\coprod_\lambda \Theta$ are two $S$-acts of cardinality~$\lambda$, and
            $\coprod_\lambda S \cong \coprod_\lambda \Theta$
            implies that $S$ must act trivially on $\coprod_\lambda S$, so $S$ is trivial.
        \end{proof}
        
    Even when $S$ does not have local zeros, $\KActp$ can be described as the union of other AECs that have joint embedding.
    This is formalized as follows.
    Define a relation $\sim$ on $\KActp$ by setting $A \sim B$ if and only if $A$ and $B$ can be jointly embedded, i.e., there are pure embeddings $f : A \rightarrow C$ and $g : B \rightarrow C$ for some $S$-act $C$.
    It is clear that $\sim$ is reflexive and symmetric, while transitivity follows from amalgamation\footnote{
        The proof of transitivity is similar to the beginning of the ``Sufficiency'' part of the proof of Proposition \ref{prop: JEP <-> local zeros} (see the diagram).
    }.
    Hence $\sim$ is an equivalence relation on $\SAct$.
    If $K'$ is a $\sim$-equivalence class of $\SAct$, then we refer to $(K',\leq_p)$ as a \emph{branch} of $\KActp$.

    \begin{fact}\label{fact: branch AECs}
        Each branch of $\KActp$ is an AEC with disjoint amalgamation, joint embedding, no maximal models, and L\"owenheim-Skolem number $\cardS+\aleph_0$.
        Moreover, there are at most $2^{\cardS+\aleph_0}$ branches of $\KActp$.
    \end{fact}
    
    We suspect that $2^{\cardS+\aleph_0}$ is usually a gross overestimate of the number of branches of $\KActp$.
    
    \begin{question}
        Is there a ``nice'' classification of all monoids $S$ for which $\KActp$ has only finitely many branches? Countably many? 
    \end{question}

    Further properties of $\KActp$ require the notion of Galois types: a semantic generalization of first-order types to AECs introduced by Shelah in \cite{shelah2006universal}.

    \begin{definition}
        Let $\K = (K ,\leq_\K)$ be an AEC.
        \begin{itemize}
            \item A \emph{pre-type} in $\K$ is a triple $(\overline{b},X,N)$ where $N \in K$, $X \subseteq N$ is a subset, and $\overline{b} \in N^{<\infty}$.
            The class of pre-types in $\K$ is denoted by $\K^3$.

            \item We define a reflexive and symmetric relation $E^\K_{at}$ on $\K^3$ by $(\overline{b}_1,X_1,N_1) E^\K_{at} (\overline{b}_2,X_2,N_2)$ if and only if $X \coloneqq X_1 = X_2$ and there exist $\K$-embeddings $f_i : N_i \rightarrow N$ with $f_i\upharpoonright X = \id_X$ for $i = 1,2$ such that $f_1(\overline{b}_1) = f_2(\overline{b}_2)$.

            \item Let $E^\K$ be the transitive closure of $E^\K_{at}$.
            For $(\overline{b},X,N) \in \K^3$, we define the \emph{Galois type of $\overline{b}$ over $X$ in $N$} to be the $E^\K$-equivalence class of $(\overline{b},X,N)$ and denote it by $\gtp_\K(\overline{b}/X;N)$.

            \item For $M \in K$ and $\alpha$ an ordinal, we write $\gS^\alpha_\K(M) = \{\gtp_\K(\overline{b}/M;N) \mid M \leq_K N, ~ \overline{b} \in N^\alpha\}$.

            \item For $p = \gtp_\K(\overline{b}/X;N)$ and $Y \subseteq X$, we write $p \upharpoonright Y = \gtp(\overline{b}/Y;N)$.
        \end{itemize}
        Usually, $\K$ will be clear from the context and omitted.
        Additionally, we write $\gS(M)$ when $\alpha = 1$.
    \end{definition}

    \begin{remark}
        Observe that $(\overline{b}_1,X_1,N_1) E^\K_{at} (\overline{b}_2,X_2,N_2)$ if and only if $X \coloneqq X_1 = X_2$ and there exist $\K$-embeddings $f_i : N_i \rightarrow N$ with $f_1\upharpoonright X = f_2 \upharpoonright X$ and $f_1(\overline{b}_1) = f_2(\overline{b}_2)$.
        This is the case as we may copy back $N$ over $f_1$ as needed.
        We will often use this equivalent description of $E^\K_{at}$.
    \end{remark}

    \begin{remark}
        For any AEC $\K$ with amalgamation, $E_{at}^\K = E^\K$.
        In particular, this is true for $\KActp$.
    \end{remark}

    In $\KActp$, we can characterize Galois types using pp-formulas.
    Recall that for a structure $N$, a subset $M \subseteq N$, and $\overline{b}\in N^{<\omega}$, the \emph{pp-type of $\overline{b}$ over $M$ in $N$} is the set of all formulas $\varphi(\overline{x},\overline{m})$, where $\varphi(\overline{x},\overline{y})$ is a pp-formula in variables $\overline{x},\overline{y}$ and $\overline{m}\in M^{\ell(\overline{y})}$ is a vector of parameters, such that $N$ satisfies $\varphi(\overline{b},\overline{m})$.
    As with pure embeddings, this can be phrased in terms of solutions to systems of equations:

    \begin{fact}
        Let $A,B,M$ be $S$-acts such that $M \subseteq A \cap B$.
        If $a \in A$ and $b \in B$, then $\pp(b/M;B) \subseteq \pp(a/M;A)$ if and only if every finite system of equations in variables $x,y_1, \dots, y_n$ with constants from $M$ and solution $b, b_1, \dots,b_n \in B$ has a solution $a,a_1, \dots, a_n \in A$.
    \end{fact}

    We are now ready to show that Galois types coincide with pp-types in $\KActp$.
    
    \begin{theorem}\label{thm: gtp = pp}
        Let $A,B,M$ be $S$-acts such that $M \leq A,B$.
        If $a \in A$ and $b \in B$, then $\gtp(a/M;A) = \gtp(b/M;B)$ if and only if $\pp(a/M;A) = \pp(b/M;B)$.
    \end{theorem}

        \begin{proof}
            The forward implication is clear, so suppose $\pp(a/M;A) = \pp(b/M;B)$.
            Without loss of generality, we assume that $A \cap B = M$.
            Let $\rho$ be the $S$-congruence on $A \cup B$ generated by $(a,b)$, and let $\pi_A : A \rightarrow (A \cup B)/\rho$ and $\pi_B : B \rightarrow (A \cup B)/\rho$ be the natural projections.
            It is clear that
            \[\begin{tikzcd}
                A \arrow[r,"\pi_A"] & (A \cup B)/\rho \\
                M \arrow[u, phantom, sloped, "\subseteq"] \arrow[r, phantom, "\subseteq"] & B \arrow[u, swap, "\pi_B"]
            \end{tikzcd}\]
            commutes and $\pi_A(a) = \pi_B(b)$, so it remains to show that $\pi_A$ and $\pi_B$ are pure embeddings.

            Towards this goal, we first explicitly characterize the $\rho$-equivalence classes.
            If $(c,d) \in \rho$ with $c \neq d$, then there exists some $0 < n < \omega$ and some $s_i \in S$ and $p_i,q_i \in A \cup B$ with $\{p_i,q_i\} = \{a,b\}$ for $1 \leq i \leq n$ such that the following system of equations holds: 
            \[
                c = s_1p_1, \quad s_iq_i = s_{i+1}p_{i+1} ~\text{for}~ 1 \leq i \leq n-1, \quad \text{and} \quad s_nq_n = d.
            \]
            Without loss of generality, suppose $n$ is chosen so that the size of such a system is minimal.
            Towards a contradiction, suppose $n \geq 2$ so that $s_1q_1 = s_2p_2$ appears in the system.
            If $q_1 = p_2$, then $p_1 = q_2$ and $(s_1x = s_2x) \in \pp(a/M;A) = \pp(b/M;B)$.
            Hence $c = s_1p_1 = s_2q_2$, contradicting the minimality of $n$.
            If $q_1 \neq p_2$, then $\{q_1,p_2\} = \{a,b\}$
            and so $s_1q_1 = s_2p_2 \in A \cap B = M$.
            Hence $(s_1x = s_2p_2) \in \pp(a/M;A) = \pp(b/M;B)$, and so $c = s_1p_1 = s_2p_2$, again contradicting the minimality of $n$.
            Thus, $n = 1$, meaning that $c = s_1p_1$ and $s_1q_1 = d$.
            It immediately follows that for every $c \in A \cup B$, we have that
            \[
                [c] = \begin{cases}
                \{c\} & \text{ if } ~ c \in (A \cup B) \setminus(Sa \cup Sb), \\
                \{sa,sb\} & \text{ if } ~ c \in \{sa,sb\} ~ \text{for some} ~ s \in S.
            \end{cases}
            \]
            In particular, $\pi_A$ and $\pi_B$ are embeddings because each equivalence class contains at most one element from $A$, resp. $B$.
            This is the case because, if $sb \in A$, then $sb \in A \cap B = M$, and so $(sx = sb) \in \pp(b/M;B) = \pp(a/M;A)$.
            Hence $sa = sb$.
            Similarly, $sa = sb$ if $sa \in B$.
    
            We only show that $\pi_A$ is a pure embedding as the argument for $\pi_B$ is similar.
            Let $\Sigma$ be a finite system of equations in variables $x_1, \dots, x_n$, $y_1, \dots, y_m$, $z_1, \dots, z_k$ with constants from $\pi_A(A)$ and solution $[a_1], \dots, [a_n] \in \pi_A(A\setminus Sa)$, $[b_1], \dots, [b_m] \in \pi_B(B \setminus Sb)$, and $[s_1a], \dots, [s_ka] \in \pi_A(Sa) = \pi_B(Sb)$.
            Let $\Delta_A$ be the subset of $\pi_A^{-1}(\Sigma)$ consisting of equations in $\Sigma$ that do not contain $y_i$ for any $1 \leq i \leq m$, where any constants have been replaced by their unique preimage under $\pi_A$.
            Observe that $a_1, \dots, a_n$, $s_1a, \dots, s_ka \in A$ is a solution to $\Delta_A$ since $\pi_A$ is injective.
            Additionally, let $\Delta_B = \pi_A^{-1}(\Sigma)\setminus \Delta_A$ consist of the equations that do contain $y_i$ for some $1 \leq i \leq m$.
            Without loss of generality, each equation in $\Delta_B$ is of one of the following forms for some $t,u \in S$ and $c \in A \setminus Sa$: $ty_i = uy_j$, $tx_i = uy_j$, $ty_i = uz_j$, $ty_i = c$, $ty_i = ua$.
            From $\Delta_B$, we build a related system of equations $\Delta_B'$ in variables $x,y_1,\dots,y_m$ as follows.
            \begin{itemize}
                \item If $(ty_i = uy_j) \in \Delta_B$, then $[tb_i] = [ub_j]$.
                Since $\pi_B$ is an embedding, we know that $tb_i = ub_j$.
                In this case, let $(ty_i = uy_j) \in \Delta_B'$.
            
                \item If $(tx_i = uy_j) \in \Delta_B$ and $ta_i = ub_j$, then $ta_i \in A \cap B = M$.
                In this case, let $(uy_j = ta_i) \in \Delta_B'$.

                \item If $(tx_i = uy_j) \in \Delta_B$ and $ta_i \neq ub_j$, then there is some $v \in S$ such that $[ta_i] = [ub_j] = \{va,vb\}$.
                Since $\pi_A$ and $\pi_B$ are embeddings, we know that $ta_i = va$ and $ub_j = vb$.
                In this case, let $(uy_j = vx) \in \Delta_B'$.

                \item If $(ty_i = uz_j) \in \Delta_B$, then $[tb_i] = [us_ja] = [us_jb]$.
                Since $\pi_B$ is an embedding, we know that $tb_i = us_jb$.
                In this case, let $(ty_i = us_jx)\in \Delta_B'$.

                \item If $(ty_i = c) \in \Delta_B$, then $[tb_i] = [c]$.
                Either $[c] = \{c\}$ or $[c] = \{va,vb\}$ for some $v \in S$.
                Since $c \notin Sa$ and $\pi_B$ is an embedding, we find that $tb_i = c \in A \cap B = M$ either way.
                In this case, let $(ty_i = c) \in \Delta_B'$.

                \item If $(ty_i = ua) \in \Delta_B$, then $[tb_i] = [ua] = [ub]$.
                Since $\pi_B$ is an embedding, we know that $tb_i = ub$.
                In this case, let $(ty_i = ux) \in \Delta_B'$.
            \end{itemize}
            By construction, $\Delta_B'$ has constants from $M$ and solution $b,b_1, \dots, b_m \in B$.
            Because $\pp(a/M;A) = \pp(b/M;B)$, there are some $c_1, \dots, c_m \in A$ such that $a,c_1, \dots, c_m$ is a solution to $\Delta_B'$ in $A$.
            However, it can be easily verified by checking each case that the construction of $\Delta_B'$ is such that this implies that $a_1, \dots,a_n$, $c_1, \dots,c_m$, $s_1a,\dots,s_ka$ is a solution to $\Delta_B$ and hence to $\pi_A^{-1}(\Sigma) = \Delta_A \cup \Delta_B$.
        \end{proof}

    Tameness is a natural property of AECs that was isolated by Grossberg and VanDieren in \cite{grossberg2006galois} and can be used as a way of measuring the possible complexity of Galois types.
    For a cardinal $\lambda$, we say that an AEC $\K = (K ,\leq_\K)$ is \emph{$(<\lambda)$-tame} if for every $M \in K$ and distinct $p,q \in \gS(M)$, there is some $X \subseteq M$ with $\card{X} < \lambda$ such that $p\upharpoonright{X} \neq q\upharpoonright{X}$.
    The following result is an immediate consequence of Theorem \ref{thm: gtp = pp}.
    
    \begin{corollary}\label{cor: tameness}
        $\KActp$ is $(<\aleph_0)$-tame.
    \end{corollary}
    
        \begin{proof}
            Let $M$ be an $S$-act and $p,q \in \gS(M)$ be distinct.
            Since $\KActp$ has amalgamation, we may assume that $p = \gtp(a/M;A)$ and $q = \gtp(b/M;A)$ for some $S$-act $A$ and $a,b \in A$.
            By Theorem \ref{thm: gtp = pp}, we know that $\pp(a/M;A) \neq \pp(b/M;A)$.
            Without loss of generality, there is some pp-formula $\varphi(x,\overline{m})$ with parameters $\overline{m}\in M^{<\omega}$ %that is satisfied in $A$ by $a$ but not by $b$.
            such that $A$ satisfies $\varphi(a,\overline{m})$ but not $\varphi(b,\overline{m})$.
            Observe that this implies that $p \upharpoonright \overline{m} \neq q \upharpoonright \overline{m}$.
        \end{proof}
        
    Stability and superstability are two of the key dividing lines in model theory and will be the primary focus of Section \ref{sec: main results}.
    We state the relevant definitions here.
            
    Let $\K$ be an AEC and $\lambda \geq \LS(\K)$ be a cardinal.
    The \emph{stability spectrum} of $\K$, written $\Stab(\K)$, is the class of cardinals $\lambda \geq \LS(\K)$ such that $\card{\gS(M)} \leq \lambda$ for every $M \in \K_{\lambda}$.\footnote{Note that for all cardinals $\lambda \geq \LS(\K)$ and $M \in \K_{\lambda}$ the much weaker estimate $\card{\gS(M)} \leq 2^\lambda$ holds.
    Thus, stability refers to a situation with only ``few'' Galois types.}
    The elements of the stability spectrum are called \emph{stability cardinals}.
    We say that $\K$ is \emph{stable} if $\Stab(\K) \neq \varnothing$, and $\K$ is \emph{stable in $\lambda$} if $\lambda \in \Stab(\K)$.
    Similarly, we say that $\K$ is \emph{superstable} if there is some cardinal $\mu \geq \LS(\K)$ such that $[\mu,\infty) \subseteq \Stab(\K)$.

    \begin{remark}\label{rem: branch stability and tameness}
        Observe that if $\K'$ is a branch of $\KActp$, then $E^{\KActp} \upharpoonright (\K')^3 = E^{\K'}$.
        Hence $\KActp$ is stable in $\lambda$ if and only if every branch is stable in $\lambda$.
        Moreover, Theorem \ref{thm: gtp = pp} holds for Galois types in $\K'$ as well.
        Thus, the argument from Corollary \ref{cor: tameness} can also be used to show that every branch is $(<\aleph_0)$-tame.
    \end{remark}
    
    For $M,N \in \K$, we say that $N$ is \emph{universal over $M$} if $M \leq_\K N$, $\card{M} = \card{N} = \lambda$, and for every $L \in \K_\lambda$ such that $M \leq_\K L$, there is a $\K$-embedding $f : L \rightarrow N$ with $f\upharpoonright M = \id_M$.
    In this case, we write $M \leq_\K^u N$.
    The existence of universal extensions follows from stability.

    \begin{fact}[{\cite[Claim 2.9]{grossberg2006galois}}]\label{fact: stable -> universal exts}
        Let $\K$ be an AEC with amalgamation.
        If $\K$ is stable in $\lambda$, then for every $M \in \K_\lambda$, there is some $N \in \K_\lambda$ that is universal over $M$.
    \end{fact}

%-----%-----%-----%
%-----%-----%-----%
%-----%-----%-----%
%-----%-----%-----%
\subsection{Independence Relations} \label{sub: prelims - Ind. Rel.}
%-----%-----%-----%
%-----%-----%-----%
%-----%-----%-----%
%-----%-----%-----%

    We make use of the category-theoretic approach to independence that was introduced in \cite{Lieberman_2019}.

    \begin{definition}[{\cite[Definitions 3.4 and 3.9]{Lieberman_2019}}]
        Let $\mathcal{C}$ be a category.
        An \emph{independence relation} $\nf$ on $\mathcal{C}$ is a collection of commuting squares $(f_1,f_2,g_1,g_2)$
        such that $(f_1,f_2,g_1,g_2) \in \nf$ whenever $(f_1,f_2,h_1,h_2) \in \nf$ and there is a commuting diagram of the following form for some morphisms $k_1$ and $k_2$.
        \[
        \begin{tikzcd}
             & C_1 \arrow[r, dashed, "k_1"] & D \\
            A \arrow[ur, "g_1"] \arrow[rr, "h_1"{pos=0.8}] & & C_2 \arrow[u, dashed, swap, "k_2"] \\
            M \arrow[u, "f_1"] \arrow[r, swap, "f_2"] & B \arrow[uu, swap, "g_2"{pos=0.8}] \arrow[ur, swap, "h_2"] & 
        \end{tikzcd}
        \]
        The elements of $\nf$ are called \emph{independent squares}.
        
        We say that $\nf$ is \emph{symmetric} when $(f_1,f_2,g_1,g_2) \in \nf$ if and only if $(f_2,f_1,g_2,g_1) \in \nf$.
        Moreover, $\nf$ is \emph{weakly stable} if it is symmetric and transitive, and has uniqueness and existence \cite[Definitions 3.15, 3.13, and 3.10]{Lieberman_2019}.\footnote{
            The details of these definitions are omitted as they will not be relevant to this paper.
        }
    \end{definition}

    An \emph{independence relation $\nf$ on an AEC $\K = (K ,\leq_\K)$} means an independence relation $\nf$ on the category $K$ with $\K$-embeddings.
    In this case, we write $\dnf{M}{A}{B}{C}$ when the inclusion maps $(\text{id}_{M \rightarrow A}, \id_{M \rightarrow B}, \id_{A\rightarrow C}, \id_{B \rightarrow C})$ form an independent square.
    
    For AECs, weakly stable independence relations can be extended to sets rather than just models and can be used to define certain properties of Galois types.
    We write $\dnfb{M}{X}{Y}{N}$ when $M \leq_\K N$, $X \cup Y \subseteq N$, and there exist $M_1,M_2,M_3 \in K$ such that $X \subseteq M_1$, $Y \subseteq M_2$, $N \leq_\K M_3$, and $\dnf{M}{M_1}{M_2}{M_3}$.
    For $M \subseteq X$, we say that $\gtp(\overline{b}/X;N)$ \emph{does not} $\nf$-\emph{fork over $M$} if $\dnfb{M}{\operatorname{ran}(\overline{b})}{X}{N}$, where $\operatorname{ran}(\overline{b})$ is the unordered set of entries in $\overline{b}$.
    When $\nf$ is clear from the context, we will just say $\gtp(\overline{b}/X;N)$ \emph{does not fork over $M$}.

    \begin{definition}[{\cite[Definitions 8.6, 8.7, and 3.24]{Lieberman_2019}}]\label{def:2.20}
        Let $\K = (K,\leq_\K)$ be an AEC with a weakly stable independence relation $\nf$.
        \begin{itemize}
            \item We say that $\nf$ has \emph{(right) local character} if for every cardinal $\mu$, there exists a cardinal $\lambda_\mu$ such that for any $N \in K$ and any $p \in \gS^\mu(N)$, there is some $M \leq_\K N$ for which $\card{M} \leq \lambda_\mu$ and $p$ does not $\nf$-fork over $M$.
            If such a cardinal $\lambda_1$ exists for $\mu = 1$, we say that $\nf$ has \emph{strong $\lambda_1$-local character}.

            \item For an infinite cardinal $\lambda$, we say that $\nf$ satisfies the \emph{(right) $(<\lambda)$-witness property} if whenever $M_0 \leq_\K M_i \leq_\K M_3$ for $i = 1,2$ such that $\dnfb{M_0}{M_1}{X}{M_3}$ for every $X \subseteq M_2$ with $\card{X} < \lambda$, then $\dnf{M_0}{M_1}{M_2}{M_3}$.
            We say that $\nf$ has the \emph{witness property} if it has the $(<\lambda)$-witness property for some infinite cardinal $\lambda$.

            \item A weakly stable independence relation is called \emph{stable}\footnote{The definition of stable independence relations given in \cite{Lieberman_2019} is different from, but equivalent to, the one given here (see \cite[Theorem 8.14]{Lieberman_2019}).} 
            if it has local character and the witness property.
        \end{itemize}
    \end{definition}

    We study the following independence relation on $\KActp$ introduced in \cite{cox2025cofibrantgenerationpuremonomorphisms}.
    
    \begin{definition} \label{def:2.21}
        Define an independence relation $\nfp$ on $\KActp$ to be the collection of pullback squares of pure embeddings.
        In the case of an independent square of inclusion maps, we have that $\dnfp{M}{A}{B}{C}$ if and only if $M = A \cap B$.
    \end{definition}

    The following result from \cite{cox2025cofibrantgenerationpuremonomorphisms} shows a connection between $\nfp$ and purity.
    Recall that a commuting square of embeddings $A \leftarrow M \rightarrow B$ and $A \rightarrow C \leftarrow B$ is called \emph{pure effective} if the relevant pushout-induced map $P \rightarrow C$ is a pure embedding.
    \[
        \begin{tikzcd}[column sep = tiny, row sep = tiny]
            A \arrow[rr] \arrow[rd] & & C \\
            & P \arrow[ur, dashed] & \\
            M \arrow[uu] \arrow[rr] & & B \arrow[uu] \arrow[lu]
        \end{tikzcd}
    \]

    \begin{fact}[{\cite[Lemma 4.4]{cox2025cofibrantgenerationpuremonomorphisms}}]\label{fact: pullback + LO -> pure effective}
        Let $A,B,C,M$ be $S$-acts and
        \[
            \begin{tikzcd}
                A \arrow[r, "\operatorname{pure}"] & C \\
                M \arrow[u] \arrow[r] & B \arrow[u, swap, sloped, "\operatorname{pure}"]
            \end{tikzcd}
        \]
        be a pullback square of inclusions such that the maps $A \rightarrow C \leftarrow B$ are pure embeddings.
        If $\operatorname{C}_M^C(A \setminus M) \cap \operatorname{C}_M^C(B \setminus M) = \varnothing$, then the square is pure effective.
    \end{fact}

    \begin{remark}
        The original statement of \cite[Lemma 4.4]{cox2025cofibrantgenerationpuremonomorphisms} differs slightly from Fact \ref{fact: pullback + LO -> pure effective} in that all maps in the commuting square are assumed to be pure embeddings.
        However, careful inspection of the proof reveals that the result holds even when the arrows $A \leftarrow M \rightarrow B$ emanating from the lower left of the commuting square are just embeddings (i.e., not necessarily pure).
    \end{remark}
    
    Using methods from category theory, \cite[Theorem 5.1]{cox2025cofibrantgenerationpuremonomorphisms} established a series of equivalent conditions that include, but are not limited to, $\nfp$ being a stable independence relation on $\KActp$ and the monoid $S$ satisfying the following property.
    \begin{definition}
        A monoid $S$ is said to be \emph{locally linearly preordered} (\emph{LO} for short) if, for every $s,t \in S$, we have $s \in St$ or $t \in Ss$.\footnote{
            The abbreviated terminology LO is adopted from \cite{mustafin1988stability}.
        }
    \end{definition}

    \begin{fact}[{\cite[Theorem 5.1]{cox2025cofibrantgenerationpuremonomorphisms}}]\label{fact: LO <-> stable ind. rel.}
        The following are equivalent.
        \begin{enumerate}
            \item[$(1)$] $S$ is an LO monoid.

            \item[$(2)$] $\nfp$ is a stable independence relation on $\KActp$.

            \item[$(3)$] $\nfp$ is the collection of pure effective squares of pure embeddings.
        \end{enumerate}
    \end{fact}
    A special case of Fact \ref{fact: pullback + LO -> pure effective} will be useful and, together with \cite[Lemma 4.6]{cox2025cofibrantgenerationpuremonomorphisms}, quickly implies the following fact.

    \begin{fact}\label{fact: LO -> pure unions}
        If $S$ is an LO monoid and $A,B \leq_p C$, then $A \cup B \leq_p C$.
    \end{fact}

        \begin{proof}
            Let $M = A \cap B$.
            If $M \neq \varnothing$, observe that $A \cup B$, together with the inclusion maps, is the pushout of $A \geq M \leq B$.
            In this case, all of the maps (including the pushout-induced map $A \cup B \rightarrow C$) in the following commutative diagram are inclusion maps.
            \[
            \begin{tikzcd}[column sep = tiny, row sep = tiny]
                A \arrow[rr, "\text{pure}"] \arrow[rd] & & C \\
                & A \cup B \arrow[ur, dashed] & \\
                M \arrow[uu] \arrow[rr] & & B \arrow[uu, swap, sloped, "\text{pure}"] \arrow[lu]
            \end{tikzcd}
            \]
            Since $S$ is LO and the outer square of the diagram above is a pullback square of embeddings, it follows from \cite[Lemma 4.6]{cox2025cofibrantgenerationpuremonomorphisms} that $\operatorname{C}_M^C(A \setminus M) \cap \operatorname{C}_M^C(B \setminus M) = \varnothing$.
            Thus, $A \cup B \leq_p C$ by Fact \ref{fact: pullback + LO -> pure effective}.

            If $M = \varnothing$, then the same arguments from \cite[Lemmata 4.4 and 4.6]{cox2025cofibrantgenerationpuremonomorphisms} go through by taking the coproduct rather than the pushout. 
            In particular, following  \cite[Lemma 4.6]{cox2025cofibrantgenerationpuremonomorphisms}, one can deduce $\operatorname{C}_\varnothing^C(A) \cap \operatorname{C}_\varnothing^C(B) = \varnothing$. Then,
            as in \cite[Lemma 4.4]{cox2025cofibrantgenerationpuremonomorphisms}, for any finite system of equations $\Sigma$ in variables $x_1,\ldots,x_n$ with constants from $A\cup B$ and solution in $C$ there exists a partition $(X_A,X_B)$ of $\{x_1,\ldots,x_n\}$ such that $\Sigma$ is partitioned into a system of equations $\Sigma_A$ in variables from $X_A$ with constants from $A$, and a system of equations $\Sigma_B$ in variables from $X_B$ with constants from $B$.
        \end{proof}

%-----%-----%-----%
%-----%-----%-----%
%-----%-----%-----%
%-----%-----%-----%
\section{Stability Cardinals} \label{sec: main results}
%-----%-----%-----%
%-----%-----%-----%
%-----%-----%-----%
%-----%-----%-----%

    In Subsection \ref{sub: main - stability}, 
    we show that LO monoids are exactly the monoids $S$ for which $\KActp$ is stable.
    In Subsection \ref{sub: main - superstability}, we assume $S$ is LO and establish some more stability cardinals for $\KActp$.
    In particular, we characterize superstability in $\KActp$ in terms of algebraic properties of $S$.
    Subsection \ref{sub: main - stability spectrum} then provides a characterization of the stability cardinals for $\KActp$ under some mild set-theoretic assumptions.

%-----%-----%-----%
%-----%-----%-----%
%-----%-----%-----%
%-----%-----%-----%
\subsection{Characterization of Stability} \label{sub: main - stability}
%-----%-----%-----%
%-----%-----%-----%
%-----%-----%-----%
%-----%-----%-----%
    
    Although Fact \ref{fact: LO <-> stable ind. rel.} implies that $\nfp$ has local character when $S$ is LO, we give a direct proof of this as the exact cardinals $\lambda_\mu$ that witness local character will be used later.

    \begin{lemma}\label{lem: local character}
        % If $S$ is an LO monoid and 
        Let $A,B,C$ be $S$-acts such that $A,B \leq_p C$.
        There are $S$-acts $A',M$ such that $A \leq_p A'$, $\dnfp{M}{A'}{B}{C}$, and $\card{M} \leq \card{A} + \cardS + \aleph_0$.
        In particular, $\nfp$ has 
        strong $(\cardS+\aleph_0)$-local character.
    \end{lemma}

        \begin{proof}
            We build $\leq_p$-increasing chains of $S$-acts $\{M_i \mid i < \omega\}$ and $\{A_i \mid i < \omega\}$ such that
            \begin{enumerate}
                \item $M_0 \leq_p B$ and $A \cup M_0 \subseteq A_0 \leq_p C$,

                \item $A_i \cap B \subseteq M_{i+1} \leq_p B$ for every $i < \omega$,

                \item $M_i \subseteq A_i \leq_p C$ for every $i < \omega$, and

                \item $\card{M_i}, \card{A_i} \leq \card{A} + \cardS +\aleph_0$ for every $i < \omega$.
            \end{enumerate}
            In this case, let $M = \bigcup_{i < \omega} M_i$ and $A' = \bigcup_{i < \omega} A_i$.
            Observe that $M = A' \cap B$ by Conditions (2) and~(3).
            Additionally, we have that $M \leq_p B$ by Condition (2), $A' \leq_p C$ by Condition (3), and $M \leq_p A'$ by coherence.
            Hence $\dnfp{M}{A'}{B}{C}$.
            Furthermore, $A \leq_p A'$ by coherence, and $\card{M} \leq \card{A} + \cardS + \aleph_0$ by Condition~(4).
            Therefore, it remains to show that the construction is well-defined, but this can be achieved using the L\"owenheim-Skolem axiom.
        \end{proof}

    \begin{remark}\label{rem: branch ind. rel.}
        If $S$ is an LO monoid and $\K' = (K',\leq_p)$ is a branch of $\KActp$, then $\nfp \upharpoonright K'$ is a stable independence relation on $\K'$ and has strong $(\cardS+\aleph_0)$-local character.
        This is the case because if $\dnfp{M}{A}{B}{C}$, then $A,B,C,M$ are all in the same branch of $\KActp$.
    \end{remark}

    Similarly, Fact \ref{fact: LO <-> stable ind. rel.} also implies that $\KActp$ is stable when $S$ is LO.
    However, the next lemma provides some additional information in the form of some stability cardinals which will be used later.

    \begin{lemma}\label{lem: LO -> Stable}
        Let $S$ be an LO monoid.
        Then $\KActp$ is stable in all cardinals $\lambda \geq \cardS + \aleph_0$ such that $\lambda^{\cardS + \aleph_0} = \lambda$.
    \end{lemma}

        \begin{proof}
            Let $B \in (\KActp)_\lambda$.
            Towards a contradiction, suppose $\card{\gS(B)} > \lambda$.
            Then we may choose pairwise distinct $\{p_i \mid i < \lambda^+\} \subseteq \gS(B)$.
            For each $i < \lambda^+$, we may choose $B_i \leq_p B$ such that $\card{B_i} = \cardS + \aleph_0$ and $p_i$ does not fork over $B_i$ by Lemma \ref{lem: local character}.
            Since $\lambda^{\cardS+\aleph_0} = \lambda$, it follows from the pigeonhole principle that there are some $I \subseteq \lambda^+$ and $M \leq_p B$ such that $\card{I} = \lambda^+$, $\card{M} = \cardS + \aleph_0$, and $p_i$ does not fork over $M$ for every $i \in I$.

            Define $\Phi : I \rightarrow \gS(M)$ by $\Phi(i) = p_i\upharpoonright M$.
            Since $\card{\gS(M)} \leq 2^{\card{M} + \aleph_0} \leq \lambda^{\cardS + \aleph_0} = \lambda$, we can apply the pigeonhole principle again to obtain distinct $i,j \in I$ such that $p_i\upharpoonright M = p_j\upharpoonright M$.
            Therefore $p_i = p_j$ by \cite[Theorem 8.5]{Lieberman_2019}, contradicting the original choice of $p_i$ and $p_j$.
        \end{proof}

    The inverse of Lemma \ref{lem: LO -> Stable} can be proven by making use of some notions related to stability in first-order model theory which we recall here (see for example \cite[Section 2]{casanovas2007stable} for further details).

    For a first-order formula $\varphi(v,w)$, a structure $N$, $d \in N$, and $X \subseteq N$, the \emph{complete $\varphi$-type of $d$ over $X$ in $N$} is
    \[
         \textbf{tp}_{\varphi(v,w)}(d/X;N) \coloneqq \{\varphi(v,x) \mid x \in X, ~ N \models \varphi(d,x)\} \cup \{\neg\varphi(v,x) \mid x \in X, ~ N \models \neg\varphi(d,x)\}.
    \]
    We write $\textbf{S}_{\varphi(v,w)}(X;N)$ to denote the set of complete $\varphi$-types over $X$ realized in $N$:
    \[
        \textbf{S}_{\varphi(v,w)}(X;N) \coloneqq \{\textbf{tp}_{\varphi(v,w)}(d/X;N) \mid d \in N\}.
    \]
    For a complete first-order theory $T$, we say that $\varphi(v,w)$ has the \emph{order property (relative to $T$)} if there exists a model $M$ of $T$ and some $b_i,c_i \in M$ for $i < \omega$ such that $M \models \varphi(b_i,c_j)$ if and only if $i \leq j < \omega$.

    \begin{fact}[{\cite[Proposition 2.10]{casanovas2007stable}}]\label{fact: Thm 5.11 Rami} % See 2.4, 2.8, and 2.10 from Casanovas Stability Lecture Notes
        Let $T$ be a complete first-order theory. 
        If $\varphi(v,w)$ has the order property, then for every $\lambda \geq \aleph_0$, there exists a model $N$ of $T$ and a subset $X \subseteq N$ with $\card{X} \leq \lambda$ such that ${\card{\mathbf{S}_{\varphi(v,w)}(X;N)} > \lambda}$.
    \end{fact}

    \begin{lemma}\label{lem: nLO -> nStable}
        Let $S$ be a monoid that is not LO.
        Then $\KActp$ is not stable.
    \end{lemma}

        \begin{proof}
            Since $S$ is not LO, there are some $s,t \in S$ such that $s \notin St$ and $t \notin Ss$.
            Fix $\lambda \geq \cardS+\aleph_0$.
            We show that there is an $S$-act $M$ with $\card{M} = \lambda$ and $\card{\gS(M)} > \lambda$.
            By \cite[Lemma 3.5]{mustafin1988stability}, there exist an $S$-act $A$ and $b_i,c_i \in A$ for $i < \omega$ such that $b_i \neq b_j$ and $c_i \neq c_j$ if $i \neq j$, and
            \[
                A \models \varphi(b_i,c_j) ~ \text{if and only if} ~ i \leq j < \omega
            \]
            where $\varphi(v,w) = \exists z(sz = v ~ \land ~ tz = w)$.
            Hence $\varphi(v,w)$ has the order property relative to $T := \text{Th}(A)$. It follows from Fact \ref{fact: Thm 5.11 Rami} that there exists a model $N$ of $T$ and a subset $X \subseteq N$ with $\card{X} \leq \lambda$ such that $\card{\textbf{S}_{\varphi(v,w)}(X;N)} > \lambda$.
            By the L\"owenheim-Skolem axiom, we may choose $M \leq_p N$ such that $X \subseteq M$ and $\card{M} = \lambda$.
            However, because $\card{\textbf{S}_{\varphi(v,w)}(X;N)} > \lambda$, there are some $\{d_i \mid i < \lambda^+\} \subseteq N$ such that $\textbf{tp}_{\varphi(v,w)}(d_i/X;N) \neq \textbf{tp}_{\varphi(v,w)}(d_j/X;N)$ if $i \neq j$.
            Since $\varphi(v,w)$ is a pp-formula and $X \subseteq M \leq_p N$, this implies that $\pp(d_i/M;N) \neq \pp(d_j/M;N)$ for every $i < j < \lambda^+$.
            Thus, by Theorem \ref{thm: gtp = pp}, we have that $\gtp(d_i/M;N) \neq \gtp(d_j/M;N)$ for $i < j <\lambda^+$.
        \end{proof}

    A characterization of stability in $\KActp$ follows immediately from the previous two lemmata.

    \begin{theorem}\label{thm: LO <-> Stable}
        $\KActp$ is stable if and only if $S$ is an LO monoid.
    \end{theorem}

    \begin{remark}
        Part of the main result from \cite{cox2025cofibrantgenerationpuremonomorphisms} says that there is a stable independence relation on $\KActp$ if and only if $S$ is an LO monoid. 
        Since stability is a weaker condition than the existence of a stable independence relation, the reverse direction of Theorem \ref{thm: LO <-> Stable} follows from \cite[Theorem 5.1]{cox2025cofibrantgenerationpuremonomorphisms}.
        However, the forward implication does not follow from \cite[Theorem 5.1]{cox2025cofibrantgenerationpuremonomorphisms} because it is not known if stability implies the existence of a stable independence relation.
    \end{remark}

    \begin{remark}\label{rem: compare stability to mustafin}
        In \cite[Theorem 3.1]{mustafin1988stability}, it was proven that every complete first-order theory of $S$-acts is stable if and only if $S$ is an LO monoid.
        However, this does not imply Theorem \ref{thm: LO <-> Stable} because it is not known if there is pp-quantifier elimination for acts.
        Instead, Theorem \ref{thm: gtp = pp} and the fact that the formula $\varphi(v,w)$ used in \cite[Lemma 3.5]{mustafin1988stability} is a pp-formula are used to establish Theorem \ref{thm: LO <-> Stable}.
    \end{remark}

%-----%-----%-----%
%-----%-----%-----%
%-----%-----%-----%
%-----%-----%-----%

\subsection{Characterization of Superstability}\label{sub: main - superstability} We provide a few additional stability cardinals and characterize superstability.

    Although the L\"owenheim-Skolem axiom guarantees that any subset of an $S$-act is contained in a ``small'' pure subact, the pure subacts obtained are quite algebraically opaque.
    The following lemma provides an alternative construction where control over cardinality is given up in exchange for more algebraic transparency.

    \begin{lemma}\label{lem: pure closure rel M}
        Let $Z$ be a set and $C,M$ be $S$-acts with $Z \subseteq C$ and $M \leq_p C$.
        Then 
        \[
            Z \subseteq \CM(Z \setminus M) \cup M \leq_p C.
        \]
    \end{lemma}

        \begin{proof}
            It is clear that $Z \subseteq \CM(Z \setminus M) \cup M \leq C$, so we show that $\CM(Z \setminus M) \cup M \leq_p C$.
            Let $\Sigma$ be a finite system of equations in variables $x_1, \dots, x_n$, $y_1, \dots, y_k$ with constants from $\CM(Z \setminus M) \cup M$ and solution $c_1, \dots, c_n \in \CM(Z\setminus M)$ and $d_1, \dots, d_k \in C \setminus \CM(Z\setminus M)$.
            Without loss of generality, each equation in $\Sigma$ is of one of the following forms for some $s,t \in S$ and $a \in \CM(Z \setminus M) \cup M$: $sx_i = tx_j$, $sx_i = ty_j$, $sy_i = ty_j$, $sx_i = a$, $sy_i = a$.
            From $\Sigma$, we build two related systems of equations $\Delta_1$ and $\Delta_2$ in variables $x_1, \dots, x_n$ and $y_1, \dots, y_k$, respectively, as follows:
            \begin{itemize}
                \item If $(sx_i = tx_j) \in \Sigma$, then let $(sx_i = tx_j) \in \Delta_1$.

                \item If $(sx_i = ty_j) \in \Sigma$, then $sc_i = td_j \in M$ as
                otherwise, $(c_i, sc_i = td_j, d_j)$ would be a connecting path witnessing that $d_j \in \CM(c_i) \subseteq \CM(Z \setminus M)$, a contradiction.
                In this case, let $(sx_i = sc_i) \in \Delta_1$ and $(ty_j = td_j) \in \Delta_2$.

                \item If $(sy_i = ty_j) \in \Sigma$, then let $(sy_i = ty_j) \in \Delta_2$.

                \item If $(sx_i = a) \in \Sigma$, then let $(sx_i = a) \in \Delta_1$.

                \item If $(sy_i = a) \in \Sigma$, then $sd_i = a \in M$ as
                otherwise, $a \in \CM(Z \setminus M)$, and so $d_i \in \CM(a) \subseteq \CM(Z \setminus M)$, a contradiction.
                In this case, let $(sy_i = a) \in \Delta_2$.
            \end{itemize}
            By construction, $c_1, \dots, c_n \in \CM(Z \setminus M)$ is a solution to $\Delta_1$.
            Moreover, $\Delta_2$ has constants from $M$ and solution $d_1, \dots, d_k \in C$.
            Because $M \leq_p C$, it follows that $\Delta_2$ has a solution $m_1, \dots, m_k \in M$.
            It can be easily verified by checking each case that the constructions of $\Delta_1$ and $\Delta_2$ are such that this implies that $c_1, \dots, c_n$, $m_1, \dots, m_k \in \CM(Z \setminus M) \cup M$ is a solution to $\Sigma$. 
        \end{proof}
    
    The following is one of the key model-theoretic results which uses in an important way the specific algebraic properties of acts.
    \begin{lemma}\label{lem: gamma in underline{kappa}}
        Let $S$ be an LO monoid. 
        If $\delta$ is a limit ordinal and $\{M_i \mid i \leq \delta\}$ is a $\leq_p$-increasing continuous chain such that $\operatorname{cf}(\delta) \geq \gamma_r(S)$, $M_\delta \leq_p C$, and $a \in C$,
        then there is some $i < \delta$ such that $\dnfbp{M_i}{a}{M_\delta}{C}$.
    \end{lemma}

        \begin{proof}
            If $a \in M_i$ for some $i < \delta$, then $\dnfp{M_i}{M_i}{M_\delta}{C}$ witnesses that $\dnfbp{M_i}{a}{M_\delta}{C}$. Thus, for the remainder of the proof, we will assume $a \in C \setminus M_\delta$.
            
            Consider $I = \{s \in S \mid sa \in M_\delta\}$.
            Since $M_\delta$ is an $S$-act, we know that $I$ is a left ideal of $S$.
            By definition of $\gamma_r(S)$, see Equation \eqref{eq:gamma_r}, there is some $T \subseteq I$ such that $\card{T} < \gamma_r(S)$ and $I = ST$.
            Because $\gamma_r(S)$ is regular and $\operatorname{cf}(\delta) \geq \gamma_r(S)$, it follows that there is some $i < \delta$ such that $Ta \subseteq M_i$.
            Hence
            $$Sa \cap M_\delta = Ia = STa \subseteq SM_i = M_i.$$
            Now, we let $A = \text{C}_{M_i}^C(Sa \setminus M_i) \cup M_i$.
            We know that $a \in \text{C}_{M_i}^C(Sa \setminus M_i)$ since $a \notin M_\delta$, while $A \leq_p C$ by Lemma \ref{lem: pure closure rel M}.
            Moreover, since $S$ is LO and $Sa \cap M_\delta \subseteq M_i$, we know that $$\text{C}^C_{M_i}(Sa\setminus M_i) \cap M_\delta \subseteq \text{C}^C_{M_i}(Sa \setminus M_i) \cap \text{C}^C_{M_i}(M_\delta \setminus M_i) = \varnothing$$
            by \cite[Lemma 4.6]{cox2025cofibrantgenerationpuremonomorphisms}.
            Thus, $A \cap M_\delta = \left(\text{C}^C_{M_i}(Sa \setminus M_i) \cap M_\delta\right) \cup M_i = M_i$, and so $\dnfp{M_i}{A}{M_\delta}{C}$.
        \end{proof}

         The next result provides additional stability cardinals.
    Both are instances of more general results that we have specialized to $\KActp$. We provide some details for the convenience of the reader since the references have some technical assumptions we will not introduce in this paper.

    \begin{fact}\label{fact: stable in successor}\
    
    \begin{enumerate}
    \item[$(1)$] (\cite[Theorem 4.5]{baldwin2006upward})
        If $\lambda \geq 2^{\cardS + \aleph_0}$ and $\KActp$ is stable in $\lambda$, then $\KActp$ is stable in $\lambda^+$.
        \item[$(2)$] (\cite[Lemma 5.5]{VASEY_2016}) Let $\delta$ be a limit ordinal and $\{\lambda_\alpha \mid \alpha < \delta\}$ be a strictly increasing chain of cardinals such that $\operatorname{cf}(\delta) \geq \gamma_r(S)$, $\lambda_0 \geq 2^{\cardS + \aleph_0}$, and $\KActp$ is stable in $\lambda_{\alpha +1}$ for every $\alpha < \delta$.
        Then $\KActp$ is stable in $\sup_{\alpha < \delta}\lambda_\alpha$.
    \end{enumerate}
    \end{fact}
    
        \begin{proof}\
        \begin{enumerate}
            \item Let $M \in (\KActp)_{\lambda^+}$.
            Towards a contradiction, suppose $\card{\gS(M)} > \lambda^+$.
            Then, we may choose pairwise distinct $\{p_i \mid i < \lambda^{++}\} \subseteq \gS(M)$.
            By repeatedly applying the L\"owenheim-Skolem axiom, choose a $\leq_p$-increasing continuous chain $\{M_\alpha \mid \alpha < \lambda^+\} \subseteq (\KActp)_\lambda$ such that $M = \bigcup_{\alpha < \lambda^+} M_\alpha$.
            Because $\operatorname{cf}(\lambda^+) =\lambda^+ > 2^{\cardS+\aleph_0}\ge \cardS^+ \ge \gamma_r(S)$, for every $i < \lambda^{++}$, we may choose some $\alpha_i < \lambda^+$ such that $p_i$ does not fork over $M_{\alpha_i}$ by Lemma \ref{lem: gamma in underline{kappa}}.
            It follows from the pigeonhole principle that there is some $I \subseteq \lambda^{++}$ and $\beta < \lambda^+$ such that $\card{I} = \lambda^{++}$ and $p_i$ does not fork over $M_\beta$ for every $i \in I$.

            Define $\Phi : I \rightarrow \gS(M_\beta)$ by $\Phi(i) = p_i \upharpoonright M_\beta$.
            Since $\KActp$ is stable in $\lambda$, we know that $\card{\gS(M_\beta)} \leq \lambda$.
            Thus, we can apply the pigeonhole principle again to obtain distinct $i,j \in I$ such that $p_i \upharpoonright M_\beta = p_j \upharpoonright M_\beta$.
            Therefore $p_i = p_j$ by \cite[Theorem 8.5]{Lieberman_2019}, contradicting the original choice of $p_i$ and $p_j$.
              
            \item The same argument as that of (1) works, but given $M \in (\KActp)_{\lambda}$ for $\lambda = \sup_{\alpha < \delta}\lambda_\alpha$, choose a $\leq_p$-increasing continuous chain $\{M_\alpha \mid \alpha < \delta\} \subseteq (\KActp)_{<\lambda}$ such that $M = \bigcup_{\alpha < \delta} M_\alpha$ and $\card{ M_{\alpha +1} } = \lambda_{\alpha +1}$ for each $\alpha < \delta$. \qedhere
        \end{enumerate}
        \end{proof}

    In order to characterize superstability, we will also make use of some arithmetic properties of the stability cardinals which follow from the next result.

    \begin{lemma}\label{lem: stab in pure -> stab in embed}
        $\Stab(\KActp) \subseteq \Stab(\K_{\SAct})$.
    \end{lemma}

        \begin{proof}
            Let $\lambda \in \Stab(\KActp)$.
            Recall that $\lambda \ge \LS(\KActp)$ and $\LS(\KActp) = \LS(\K_{\SAct})$. 
            Towards a contradiction, suppose $\K_{\SAct} = (\SAct,\leq)$ is not stable in $\lambda$.
            Then we may choose $M \in (\K_\SAct)_\lambda$ and pairwise distinct $\{p_i = \gtp_{\K_\SAct}(b_i/M;B_i) \mid i < \lambda^+\} \subseteq \gS_{\K_\SAct}(M)$.
            Since $\K_\SAct$ has amalgamation, we may assume that $B_i = B_0 \eqqcolon B$ for all $i < \lambda^+$.
            Let $A$ be the $S$-act obtained by applying L\"owenheim-Skolem to $M \subseteq B$ in $\KActp$.
            By the stability of $\KActp$ in $\lambda$, there are distinct $i,j < \lambda^+$ with $\gtp_{\KActp}(b_i/A;B) = \gtp_{\KActp}(b_j/A;B)$.
            Thus, there exists an $S$-act $C$ and some pure embeddings $f_i,f_j : B \rightarrow C$ such that $f_i(b_i) = f_j(b_j)$ and the following diagram commutes.
            \[
                \begin{tikzcd}
                    B \arrow[r,"f_i"] & C \\
                    A \arrow[u,sloped, phantom, "\leq_p"] \arrow[r,phantom,"\leq_p"] & B \arrow[u,swap,"f_j"]
                \end{tikzcd}
            \]
            However, since $M \subseteq A$, this diagram also witnesses that $\gtp_{\K_\SAct}(b_i/M;B) = \gtp_{\K_\SAct}(b_j/M;B)$, contradicting the choice of $i \neq j < \lambda^+$.
        \end{proof}
        
        \begin{remark}\label{rem:3.13}
        Observe that the previous result can be extended to the following general setup which will not be used in this paper:

        Assume $(K, \leq)$ and $(K, \leq^*)$ are both AECs such that $(K, \leq)$ has amalgamation.
        If $M \leq^* N$ implies $M \leq N$ for all $M, N \in K$, then $\Stab(K,\leq^*) \subseteq \Stab(K,\leq)$.
        \end{remark}

    \begin{corollary}\label{cor: stab in pure -> lambda^(<gamma) = lambda}
        If $\KActp$ is stable in $\lambda\geq 2^{\cardS+\aleph_0}$, then $\operatorname{cf}(\lambda) \geq \gamma_r(S)$.
    \end{corollary}

        \begin{proof}
            For $\lambda \geq 2^{\cardS+\aleph_0}$, it was proven in \cite[Theorem 3.1]{mazariarmida2025abstractelementaryclassacts} that  $\K_{\SAct}$ is stable in $\lambda$ if and only if $\lambda^{<\gamma_r(S)} = \lambda$.
            Therefore, the result is an immediate consequence of Lemma \ref{lem: stab in pure -> stab in embed} and K\"onig's Theorem.
        \end{proof}
 
    \begin{theorem}\label{thm: WO <-> Superstable}
        $\KActp$ is superstable if and only if $S$ is weakly noetherian and LO.
    \end{theorem}

        \begin{proof}
            \noindent\underline{Sufficiency:} Suppose $S$ is weakly noetherian and LO.
            We prove by induction that $\KActp$ is stable in $\lambda$ for all $\lambda \geq 2^{\cardS + \aleph_0}$. The base case follows from Lemma \ref{lem: LO -> Stable} and the successor step follows from Fact \ref{fact: stable in successor}(1).           
            So suppose $\lambda > 2^{\cardS + \aleph_0}$ is a limit cardinal and $\KActp$ is stable in $\mu$ for every $\lambda > \mu \geq 2^{\cardS+\aleph_0}$. 
            Then, we may choose a strictly increasing sequence of cardinals $\{\lambda_\alpha \mid \alpha < \operatorname{cf}(\lambda)\}$ such that $\lambda_\alpha < \lambda$ for every $\alpha < \operatorname{cf}(\lambda)$, $\lambda_0 \geq 2^{\cardS + \aleph_0}$, and $\sup_{\alpha < \operatorname{cf}(\lambda)}\lambda_\alpha = \lambda$.  Since $S$ is weakly noetherian, we know $\operatorname{cf}(\lambda) \geq \aleph_0 = \gamma_r(S)$. Thus, $\KActp$ is stable in $\lambda$ by Fact \ref{fact: stable in successor}(2) and the induction hypothesis.
            
            \noindent\underline{Necessity:} Suppose $\KActp$ is superstable, so that $[\mu,\infty) \subseteq \Stab(\KActp)$ for some cardinal $\mu \geq 2^{\cardS + \aleph_0}$.
            In particular, $\KActp$ is stable, and so $S$ is LO by Theorem \ref{thm: LO <-> Stable}.
            Now, choose $\lambda \geq \mu$ such that $\operatorname{cf}(\lambda) = \aleph_0$.
            Then, $\KActp$ is stable in $\lambda$, and so $\aleph_0 = \operatorname{cf}(\lambda) \geq \gamma_r(S)$ by Corollary \ref{cor: stab in pure -> lambda^(<gamma) = lambda}.
            However, as $\gamma_r(S) \geq \aleph_0$ by definition, it follows that $S$ is also weakly noetherian.
        \end{proof}

    \begin{remark}
        Like with Theorem \ref{thm: LO <-> Stable}, there is a first-order analogue of Theorem \ref{thm: WO <-> Superstable}.
        In \cite[Theorem~4.1]{mustafin1988stability}, it was proven that every complete first-order theory of $S$-acts is superstable if and only if $S$ is weakly noetherian and LO.\footnote{
            This combination of properties is called being WO in \cite{mustafin1988stability}.
        }
        However, just as in Remark \ref{rem: compare stability to mustafin}, this does not imply Theorem \ref{thm: WO <-> Superstable} because it is not known if there is pp-quantifier elimination for acts.
    \end{remark}

%-----%-----%-----%
%-----%-----%-----%
%-----%-----%-----%
%-----%-----%-----%
\subsection{Stability Spectrum}\label{sub: main - stability spectrum}
%-----%-----%-----%
%-----%-----%-----%
%-----%-----%-----%
%-----%-----%-----%

    Under some set-theoretic hypotheses, we completely characterize the stability spectrum of $\KActp$.
    Note that $\lambda^{<\gamma_r(S)} = \lambda$ if and only if $\lambda^{<\gamma(S)} = \lambda$ for $\lambda \geq 2^{\cardS + \aleph_0}$.\footnote{
            There is nothing to show when $\gamma(S)$ is regular.
            For $\gamma(S)$ singular, $\lambda^{<\gamma(S)}=\lambda$ gives $\lambda^{\gamma(S)} =(\lambda^{<\gamma(S)})^{\operatorname{cf}(\gamma(S))}=\lambda^{\operatorname{cf}(\gamma(S))}=\lambda$ as in \cite[Theorem 5.16]{jech2003set}.
        }
    Because of this, we state the next two results with the more algebraically meaningful $\gamma(S)$, even though the proofs use~ $\gamma_r(S)$.

    For infinite cardinals $\lambda$ and $\mu$, we say that $\mu$ is \emph{almost $\lambda$-closed} if $\nu^\lambda \leq \mu$ for all $\nu < \mu$.
    Moreover, if such a cardinal exists, we let $\theta(\lambda)$ denote the least cardinal $\theta$ such that every $\mu \geq \theta$ is almost $\lambda$-closed.

    \begin{theorem}\label{thm: GCH stability}
        Assume that $\theta(\cardS+\aleph_0)$ exists.
        If $S$ is an LO monoid, then, for any $\lambda > \theta(\cardS+\aleph_0)$, $\KActp$ is stable in $\lambda$ if and only if $\lambda^{<\gamma(S)} = \lambda$.
    \end{theorem}

        \begin{proof}
            \noindent\underline{Sufficiency:} Let $\lambda \geq \theta(\cardS+\aleph_0)\ge 2^{\cardS+\aleph_0}$. Then $\lambda \leq \lambda^{\cardS+\aleph_0} \leq \lambda^+$ as $\lambda^+$ is almost $(\cardS+\aleph_0)$-closed. If $\lambda^{\cardS+\aleph_0}= \lambda$, then $\KActp$ is stable in $\lambda$ by Lemma \ref{lem: LO -> Stable} and stable in $\lambda^+$ by Fact \ref{fact: stable in successor}(1). Otherwise, if $\lambda^{\cardS+\aleph_0}= \lambda^+$, then $(\lambda^+)^{\cardS+\aleph_0}= \lambda^+$, and $\KActp$ is stable in $\lambda^+$ by Lemma \ref{lem: LO -> Stable}. We conclude that 
            \begin{align*}
            \mbox{$\KActp$ is stable in all successor cardinals $\lambda > \theta(\cardS+\aleph_0)$.}\tag{3.1} \label{eq:3.1}
            \end{align*}
            Let now $\lambda > \theta(\cardS+\aleph_0)$ be such that $\lambda^{<\gamma_r(S)} = \lambda$.
            If $\lambda$ is a successor cardinal, then $\KActp$ is stable in $\lambda$ by Statement \eqref{eq:3.1}.
            So suppose $\lambda > \theta(\cardS+\aleph_0)$ is a limit cardinal.
            Then we may choose a strictly increasing sequence of successor cardinals $\{\lambda_\alpha \mid \alpha < \operatorname{cf}(\lambda)\}$ such that $\theta(\cardS+\aleph_0) < \lambda_\alpha < \lambda$ for every $\alpha < \operatorname{cf}(\lambda)$ and $\sup_{\alpha < \operatorname{cf}(\lambda)}\lambda_\alpha = \lambda$.
            Since $\lambda^{<\gamma_r(S)} = \lambda$, we know that $\operatorname{cf}(\lambda)\geq\gamma_r(S)$ by K\"onig's Theorem.
            Moreover, $\KActp$ is stable in $\lambda_\alpha$ for all $\alpha < \operatorname{cf}(\lambda)$ by Statement \eqref{eq:3.1}.
            Therefore $\KActp$ is stable in $\lambda$ by Fact \ref{fact: stable in successor}(2).
            
            \noindent\underline{Necessity:} Suppose $\KActp$ is stable in $\lambda$.
            Then $\K_{\SAct}$ is stable in $\lambda$ by Lemma \ref{lem: stab in pure -> stab in embed}, and
            \cite[Theorem~3.1]{mazariarmida2025abstractelementaryclassacts} gives $\lambda^{<\gamma(S)} = \lambda$.
        \end{proof}

        Recall that the \emph{Generalized Continuum Hypothesis (GCH)} says that $2^\lambda = \lambda^+$ for every infinite cardinal~$\lambda$.

    \begin{corollary} \label{cor: old GCH stability}
        Assume GCH.
        Let $S$ be an LO monoid and $\lambda \ge (\cardS+\aleph_0)^+$.
        Then, $\KActp$ is stable in $\lambda$ if and only if $\lambda^{<\gamma(S)} = \lambda$.
    \end{corollary}

        \begin{proof}
            Assuming GCH, we have that $\theta(\lambda) = 2^\lambda = \lambda^+$ for all infinite cardinals $\lambda$.
            Thus, for any $\lambda > (\cardS+\aleph_0)^+ = 2^{\cardS+\aleph_0}$, $\KActp$ is stable in $\lambda$ if and only if $\lambda^{<\gamma(S)} = \lambda$ by Theorem \ref{thm: GCH stability}, while stability in $2^{\cardS+\aleph_0}$ follows from Lemma \ref{lem: LO -> Stable}.
        \end{proof}

    \begin{remark}
        As mentioned in \cite[Fact 4.20]{vasey2018toward}, assuming either the Singular Continuum Hypothesis or the existence of a strongly compact cardinal above $\cardS+\aleph_0$ is enough to guarantee the existence of $\theta(\cardS+\aleph_0)$ in Theorem \ref{thm: GCH stability}.
    \end{remark}

        We believe that Theorem \ref{thm: GCH stability} might actually be obtainable in ZFC, but do not know how to show it.

    \begin{question}
        In Theorem \ref{thm: GCH stability}, can we drop the requirement that $\theta(\cardS+\aleph_0)$ exists?
    \end{question}
    
%-----%-----%-----%
%-----%-----%-----%
%-----%-----%-----%
%-----%-----%-----%
\section{Pure Injectivity and Limit Models} \label{sec: limit models}
%-----%-----%-----%
%-----%-----%-----%
%-----%-----%-----%
%-----%-----%-----%

    In Subsection \ref{sub: limit models - pure injectivity}, we obtain a Baer-like criterion for pure injective acts when $S$ is LO.
    This allows us to recover a result from \cite{cox2025cofibrantgenerationpuremonomorphisms} using the model-theoretic notion of limit models.
    In Subsection \ref{sub: limit models - spectrum}, we continue to study limit models in $\KActp$.
    Under some mild set-theoretic assumptions, similar to those in Subsection~\ref{sub: main - stability spectrum}, we characterize the spectrum of limit models in a particular branch of $\KActp$. 

%-----%-----%-----%
%-----%-----%-----%
%-----%-----%-----%
%-----%-----%-----%
\subsection{Pure Injective Acts} \label{sub: limit models - pure injectivity}
%-----%-----%-----%
%-----%-----%-----%
%-----%-----%-----%
%-----%-----%-----%

    To begin, we establish a result (Lemma \ref{lem: independent chains}) about building increasing continuous chains that are, in some sense, independent with respect to $\nfp$.
    Our proof makes use of the following two properties of $\nfp$.

    \begin{lemma}[Base Monotonicity]\label{lem: base monotonicity}
        Let $S$ be an LO monoid, $C,M,N$ be $S$-acts, and $X,Y$ be sets such that $\dnfbp{M}{X}{Y}{C}$ and $M \leq_p N \leq_p C$.
        If $N \subseteq X$ or $N \subseteq Y$, then $\dnfbp{N}{X}{Y}{C}$.
    \end{lemma}

        \begin{proof}
            Since $\nfp$ is symmetric, we may assume without loss of generality that $N \subseteq Y$.
            By assumption, we may choose $S$-acts $A,B,D$ such that $X \subseteq A$, $Y \subseteq B$, $C \leq_p D$, and $\dnfp{M}{A}{B}{D}$.
            Now $(A \cup N) \cap B = N$ since $A \cap B = M$ and $N \subseteq B$.
            Moreover, $A \cup N \leq_p D$ by Fact \ref{fact: LO -> pure unions} because $A \leq_p D$ and $N \leq_p C \leq_p D$.
            Therefore $\dnfp{N}{(A\cup N)}{B}{D}$ witnesses that $\dnfbp{N}{X}{Y}{C}$.
        \end{proof}

\begin{remark}
        Every weakly stable independence relation has base monotonicity (see \cite[Lemma 3.23]{Lieberman_2019}), so Lemma \ref{lem: base monotonicity} also follows from Fact \ref{fact: LO <-> stable ind. rel.}.    
\end{remark}

    When $S$ is LO, Fact \ref{fact: LO <-> stable ind. rel.} also implies that $\nfp$ has the $(<\lambda)$-witness property for some infinite cardinal~$\lambda$ (see Definition \ref{def:2.20}).
    Our next lemma shows that we have $\lambda = \aleph_0$.

    \begin{lemma}[$(<\aleph_0)$-Witness Property]\label{lem: witness property}
        Let $A,B,C,M$ be $S$-acts such that $M \leq_p A \leq_p C$ and $M \leq_p B \leq_p C$.
        If $\dnfbp{M}{A}{X}{C}$ for every finite $X \subseteq B$, then $\dnfp{M}{A}{B}{C}$.
    \end{lemma}

        \begin{proof}
            Since $M \subseteq A \cap B$, it suffices to let $d \in A \cap B$ and show that $d \in M$.
            By assumption, we have that $\dnfbp{M}{A}{\{d\}}{C}$.
            Thus, there are $S$-acts $A',B',C'$ such that $A \subseteq A'$, $d \in B'$, $C \leq_p C'$, and $\dnfp{M}{A'}{B'}{C'}$.
            In particular, $d \in A \cap B' \subseteq A' \cap B' = M$.
        \end{proof}

    The following two results are similar to \cite[Lemma 4.6 and Theorem 4.8]{MazariRosickyRelativeInjective} for the setup of $R$-modules, but we provide the details for completeness and to show that the proofs go through in the non-additive setup of $S$-acts.
    
    \begin{lemma}\label{lem: independent chains}
        Let $S$ be an LO monoid and $B,C$ be $S$-acts with $B \leq_p C$ and $\card{C} = \lambda > \cardS + \aleph_0$.
        There are $\leq_p$-increasing continuous chains $\{B_i \mid  i < \lambda\}$ and $\{C_i \mid i < \lambda\}$ such that
        \begin{enumerate}
            \item[$(1)$] $B_i \leq_p C_i \leq_p C$ and $B_i \leq_p B$ for every $i < \lambda$,

            \item[$(2)$] $\card{B_i}, \card{C_i} < \lambda$ for every $i < \lambda$,

            \item[$(3)$] $\dnfp{B_i}{C_i}{B_j}{C_j}$ for every $i < j < \lambda$, and

            \item[$(4)$] $B = \bigcup_{i < \lambda}B_i$ and $C = \bigcup_{i < \lambda}C_i$.
        \end{enumerate}
    \end{lemma}

        \begin{proof}
            Let $C = \{c_i \mid i < \lambda\}$ be an enumeration.
            We build the $\leq_p$-increasing continuous chains $\{B_i \mid i < \lambda\}$ and $\{C_i \mid i < \lambda\}$ by recursion such that
            \begin{enumerate}
                \item[$(1)\ $] $B_i \leq_p C_i \leq_p C$ and $B_i \leq_p B$ for every $i < \lambda$,

                \item[$(2)'$] $\card{B_i},\card{C_i} \leq \cardS + \aleph_0 + \card{i}$ for every $i < \lambda$,

                \item[$(3)'$] $\dnfp{B_i}{C_i}{B}{C}$ for every $i < \lambda$, and

                \item[$(4)'$] $c_i \in C_{i+1}$ and, if $c_i \in B$, then $c_i \in B_{i+1}$.
            \end{enumerate}
            In this case, it is straightforward to see that Conditions (1), (2), and (4) hold.
            Moreover, notice that Condition~$(3)'$ implies that $C_i \cap B_j = (C_i \cap B) \cap B_j = B_i \cap B_j = B_i$ for all $i < j < \lambda$.
            Thus, since $B_i \leq_p C_i \leq_p C_j$ and $B_i \leq_p B_j \leq_p C_j$ by Condition (1), it follows that Condition (3) holds as well.
            Therefore, it remains to show that this construction can be carried out. We will use transfinite recursion.

            For the base case, we let $B_0$ be the $S$-act obtained by applying the L\"owenheim-Skolem axiom to $\varnothing \subseteq B$, and let $C_0 = B_0$.
            In the case of a limit ordinal $j$, we let $B_j := \bigcup_{i<j}B_i$ and $C_j := \bigcup_{i<j}C_i$ and observe that Condition $(3)'$ holds: we have $\dnfbp{B_j}{C_i}{B}{C}$ by base monotonicity, and $\dnfp{B_j}{C_j}{B}{C}$ by symmetry and $(<\aleph_0)$-witness property.
            So consider the case of a successor ordinal $j+1 < \lambda$ and assume that $B_j,C_j$ are defined and satisfy Conditions $(1)$-$(4)'$.
            We build $\leq_p$-increasing chains $\{M_k \mid k < \omega\}$ and $\{N_k \mid k <\omega\}$ such that
            \begin{enumerate}[label = (\roman*)]
                \item[$(0)^*$] $B_j \leq_p M_0$ and $C_j \leq_p N_0$,

                \item[$(1)^*$] $M_k \leq_p N_k \leq_p C$ and $M_k \leq_p B$ for every $k < \omega$,

                \item[$(2)^*$] $\card{M_k}, \card{N_k} \leq \cardS + \aleph_0 + \card{j+1}$ for every $k < \omega$,

                \item[$(3)^*$] $\dnfbp{M_{k+1}}{N_k}{B}{C}$ for every $k < \omega$, and
                
                \item[$(4)^*$] $c_j \in N_0$ and, if $c_j \in B$, then $c_j \in M_0$.
            \end{enumerate}
            In this case, we let $B_{j+1} = \bigcup_{k < \omega}M_k$ and $C_{j+1} = \bigcup_{k < \omega}N_k$.
            It is clear that Conditions $(1)$, $(2)'$, and $(4)'$ hold while Condition $(3)'$ again holds by  base monotonicity, symmetry, and the $(<\aleph_0)$-witness property.
            Thus, it remains to show that this secondary construction can be carried out.

            Let $M_0$ be the $S$-act obtained by applying L\"owenheim-Skolem to $B_j \cup (B \cap \{c_j\}) \subseteq B$.
            Similarly, let $N_0$ be the $S$-act obtained by applying L\"owenheim-Skolem to $M_0 \cup C_j \cup \{c_j\} \subseteq C$.
            Note that $M_0$ and $N_0$ are as desired by coherence.
            Now, suppose that $k \geq 0$ and $M_k,N_k$ are defined and satisfy Conditions $(0)^*$-$(4)^*$.
            By Lemma \ref{lem: local character}, we may choose $S$-acts $M,A$ such that $N_k \leq_p A$, $\dnfp{M}{A}{B}{C}$, and $\card{M} \leq \card{N_k} + \cardS+\aleph_0$.
            Let $M_{k+1}$ be the $S$-act obtained by applying L\"owenheim-Skolem to $M \cup M_k \subseteq B$.
            Now $\dnfp{M}{A}{B}{C}$ witnesses that $\dnfbp{M}{N_k}{B}{C}$, and so $\dnfbp{M_{k+1}}{N_k}{B}{C}$ by base monotonicity and coherence, i.e., Condition $(3)^*$ holds.
            Lastly, let $N_{k+1}$ be the $S$-act obtained by applying L\"owenheim-Skolem to $N_k \cup M_{k+1} \subseteq C$. It is easy to check that $M_{k+1}, N_{k+1}$ satisfy Conditions $(1)^*$ and $(2)^*$.
        \end{proof}

    We recall the definition of pure injectivity, then prove a Baer-like criterion for pure injectivity in $\SAct$.

    \begin{definition}
        An $S$-act $A$ is \emph{pure injective} if, for any $S$-acts $B \leq_p C$, every $S$-homomorphism $f : B \rightarrow A$ can be extended to an $S$-homomorphism from $C$ to $A$, i.e., there is an $S$-homomorphism $g : C \rightarrow A$ such that $g\upharpoonright B = f$.
    \end{definition}

    \begin{theorem}\label{thm: Baer-like for p.i.}
        Let $S$ be an LO monoid.
        An $S$-act $A$ is pure injective if and only if, for any $S$-acts $B \leq_p C$ such that $\card{C} \leq \cardS + \aleph_0$, every $S$-homomorphism $f : B \rightarrow A$ can be extended to an $S$-homomorphism from $C$ to $A$.
    \end{theorem}

        \begin{proof}
            The forward direction is clear, so we prove the reverse direction.
            We show by induction on $\lambda \geq \cardS + \aleph_0$ that for every $B \leq_p C$ with $\card{C} \leq \lambda$ and every $S$-homomorphism $f : B \rightarrow A$, there is an $S$-homomorphism $g : C \rightarrow A$ extending $f$.

            The base case $\lambda = \cardS + \aleph_0$ is given by our assumption, so we do the induction step.
            Let $B \leq_p C$ with $\card{C} \leq \lambda$, and let $f : B \rightarrow A$ be an $S$-homomorphism.
            If $\card{C} < \lambda$, then we can extend $f$ by the induction hypothesis.
            So, without loss of generality, we assume that $\card{C} = \lambda > \cardS + \aleph_0$.
            Since $B \leq_p C$, we may choose $\{B_i \mid i < \lambda\}$ and $\{C_i \mid i < \lambda\}$ to be as in Lemma \ref{lem: independent chains}.
            For every $i < \lambda$, let $f_i = f\upharpoonright {B_i} : B_i \rightarrow A$.
            We build $\{g_i \mid  i < \lambda\}$ by transfinite recursion such that
            \begin{enumerate}
                \item $g_i : C_i \rightarrow A$ is an $S$-homomorphism for every $i < \lambda$,

                \item $f_i \subseteq g_i$ for every $i < \lambda$, and

                \item $g_i \subseteq g_j$ for every $i < j < \lambda$. 
            \end{enumerate}
            In this case, $g := \bigcup_{i < \lambda}g_i : C \rightarrow A$ is an $S$-homomorphism by Conditions (1) and (3) that extends $f$ by Condition (2).
            It remains to show that this construction can be carried out.
            
            For the base case, we may apply the induction hypothesis to $f_0 : B_0 \rightarrow A$ to obtain an $S$-homomorphism $g_0 : C_0 \rightarrow A$ such that $g_0\upharpoonright {B_0} = f_0$.
            In the limit step, we take unions as usual, $g_j := \bigcup_{i < j}g_i$, so consider the case of a successor ordinal $j + 1 < \lambda$.
            Let 
            $C_j \xrightarrow{h} P \xleftarrow{h'} B_{j+1}$ be the pushout of $C_j \xleftarrow{\id} B_j \xrightarrow{\id} B_{j+1}$,
            and let $k : P \rightarrow C_{j+1}$ be the pushout-induced map of $C_j \xrightarrow{\id} C_{j+1} \xleftarrow{\id} B_{j+1}$.
            \begin{center}
                \begin{tikzcd}
                     & & C_{j+1} \\
                    C_j \arrow[r,"h"] \arrow[urr, bend left = 20, "\text{id}"] & P \arrow[ur, dashed, "k"] & \\
                    B_j \arrow[u, "\text{id}"] \arrow[r, "\text{id}"] & B_{j+1} \arrow[u, "h'"] \arrow[uur, bend right = 20, swap, "\text{id}"] &
                \end{tikzcd}
            \end{center}
            Since $\dnfp{B_j}{C_j}{B_{j+1}}{C_{j+1}}$, we know that $k$ is a pure embedding by Fact \ref{fact: LO <-> stable ind. rel.}.
            Similarly, since $g_j\upharpoonright {B_j} = f_j = f_{j+1}\upharpoonright {B_j}$, we also obtain the pushout-induced map $k' : P \rightarrow A$ of $C_j \xrightarrow{g_j} A \xleftarrow{f_{j+1}} B_{j+1}$.
            \begin{center}
                \begin{tikzcd}
                     & & A \\
                    C_j \arrow[r,"h"] \arrow[urr, bend left = 20, "g_j"] & P \arrow[ur, dashed, "k'"] & \\
                    B_j \arrow[u, "\text{id}"] \arrow[r, "\text{id}"] & B_{j+1} \arrow[u, "h'"] \arrow[uur, bend right = 20, swap, "f_{j+1}"] &
                \end{tikzcd}
            \end{center}
            Because $\card{C_{j+1}} < \lambda$ and $k : P \rightarrow C_{j+1}$ is a pure embedding, it follows from the induction hypothesis on $A$ that there exists an $S$-homomorphism $g_{j+1} : C_{j+1} \rightarrow A$ such that $g_{j+1} \circ k = k'$.
            \begin{center}
                \begin{tikzcd}
                    P \arrow[r,"k"] \arrow[d, swap, "k'"] & C_{j+1} \arrow[dl, dashed, "g_{j+1}"] \\
                    A & 
                \end{tikzcd}
            \end{center}
            Observe that, if $b \in B_{j+1}$, then
            $$g_{j+1}(b) = g_{j+1}(k(h'(b))) = k'(h'(b)) = f_{j+1}(b).$$
            Thus, $f_{j+1} \subseteq g_{j+1}$.
            Similarly, if $c \in C_j$, then
            $$g_{j+1}(c) = g_{j+1}(k(h(c))) = k'(h(c)) = g_j(c).$$
            Thus, $g_j \subseteq g_{j+1}$ as well.
        \end{proof}

    Limit models are a major object of study in model theory that were introduced in \cite{Shelah1996}.
    We recall the relevant definitions here.

    \begin{definition}
        Let $\K$ be an AEC, $\lambda \geq \LS(\K)$ be a cardinal, and $\delta < \lambda^+$ be a limit ordinal.
        We say that $N \in \K$ is a \emph{$(\lambda,\delta)$-limit model over $M\in \K_\lambda$} if there is a $\leq_\K$-increasing continuous chain $\{M_i \mid i < \delta\} \subseteq \K_\lambda$ such that:
        \begin{itemize}
            \item $M = M_0$, $N = \bigcup_{i < \delta}M_i$, and

            \item $M_{i+1}$ is universal over $M_i$ for every $i < \delta$.
        \end{itemize}
        We say that $N$ is a \emph{$(\lambda,\delta)$-limit model} if $N$ is a $(\lambda,\delta)$-limit model over $M$ for some $M \in \K_\lambda$.
        Similarly, we say that $N$ is a \emph{$\lambda$-limit model} if $N$ is a $(\lambda,\delta)$-limit model for some limit ordinal $\delta < \lambda^+$, and that $N$ is a \emph{limit model} if $N$ is a $\lambda$-limit model for some cardinal $\lambda \geq \LS(\K)$.
    \end{definition}

    Theorem \ref{thm: Baer-like for p.i.} can be used to show that long limit models are pure injective.

    \begin{corollary}\label{cor: long limits are p.i.}
        Let $S$ be an LO monoid.
        Then, in $\KActp$,
        every $(\lambda,\delta)$-limit model with $\operatorname{cf}(\delta) \geq (\cardS + \aleph_0)^+$ is pure injective.
    \end{corollary}

        \begin{proof}
            Let $\{A_i \mid i < \delta\}$ be a witness of $A$ being a $(\lambda,\delta)$-limit model with $\operatorname{cf}(\delta) \geq (\cardS+\aleph_0)^+$.
            Let $B,C$ be $S$-acts such that $B \leq_p C$, and let $f : B \rightarrow A$ be an $S$-homomorphism.
            By Theorem \ref{thm: Baer-like for p.i.}, we may assume that $\card{C} \leq \cardS + \aleph_0$.

            Since $\card{f(B)} \leq \card{B} \leq \card{C} \leq \cardS + \aleph_0 < \operatorname{cf}(\delta)$ and $f(B) \subseteq A = \bigcup_{j < \delta}A_j$, there is some $i < \delta$ such that $f(B) \subseteq A_i$.
            Let $A_i \xrightarrow{\id} P \xleftarrow{g_1} C$ be the pushout of $A_i \xleftarrow{f} B \xrightarrow{\id} C$.
            We know that $A_i \leq_p P$ by Fact \ref{fact: preserving purity}(3).
            Moreover, observe that $\card{P} = \lambda$ because $\lambda = \card{A_i} \leq \card{P} \leq \card{A_i} + \card{C} \leq \lambda + (\cardS + \aleph_0) = \lambda$.
            Thus, since $A_{i+1}$ is universal over $A_i$, there exists a pure embedding $g_2 : P \rightarrow A_{i+1}$ such that $g_2\upharpoonright {A_i} = \text{id}_{A_i}$.
            \begin{center}
                \begin{tikzcd}
                    C \arrow[r, "g_1"] & P \arrow[dr, dashed, "g_2"] & \\
                    B \arrow[u, "\id"] \arrow[r, "f"] & A_i \arrow[u, phantom, sloped, "\leq_p"] \arrow[r, phantom, "\leq_p^u"] & A_{i+1}
                \end{tikzcd}
            \end{center}
            Observe that $g \coloneqq g_2 \circ g_1 : C \rightarrow A_{i+1} \subseteq A$ is an $S$-homomorphism such that $g \upharpoonright B = f$.
            Therefore $A$ is pure injective.
        \end{proof}
    
    Recall that a category has \emph{enough pure injectives} if every object has a pure monomorphism into a pure injective object.
    Although it was already shown as a special case of \cite[Corollary 1.3]{cox2025cofibrantgenerationpuremonomorphisms}\label{cor: enough p.i.} that $\SAct$ has enough pure injectives when $S$ is an LO monoid, Corollary \ref{cor: long limits are p.i.}
    provides a more direct proof of this.

    \begin{fact} \label{fact: enough p.i.}
        If $S$ is an LO monoid, then $\SAct$ has enough pure injectives.
    \end{fact}

        \begin{proof}
            Let $A$ be an $S$-act and $\lambda  = 2^{\card{A} + \cardS + \aleph_0}$.
            Note that $\KActp$ is stable in $\lambda$ by Lemma \ref{lem: LO -> Stable}.
            Since $\KActp$ has no maximal models, we may choose  $B \in (\KActp)_\lambda$ with $A \leq_p B$.
            Moreover, by repeatedly applying Fact \ref{fact: stable -> universal exts},  there is some $(\lambda,(\cardS+\aleph_0)^+)$-limit model $C$ over $B$ in $\KActp$.
            Then $A \leq_p C$ and, by Corollary~\ref{cor: long limits are p.i.}, $C$ is pure injective.
        \end{proof}

%-----%-----%-----%
%-----%-----%-----%
%-----%-----%-----%
%-----%-----%-----%
\subsection{Spectrum of Limit Models}\label{sub: limit models - spectrum}
%-----%-----%-----%
%-----%-----%-----%
%-----%-----%-----%
%-----%-----%-----% 
Let $\K_0$ denote the branch of $S$-acts containing the zero act $\Theta$. In this subsection, we study the spectrum of limit models of the branch $\K_0$. The reason we focus on a single branch instead of on $\KActp$ is because limit models in different branches cannot be isomorphic, while we focus on the branch $\K_0$ because of the following result:

    \begin{lemma}\label{lem: stability spectrum of zero branch}
        $\Stab(\K_0) = \Stab(\KActp)$.
    \end{lemma}

        \begin{proof}
            It suffices to show that $\Stab(\K_0) \subseteq \Stab(\KActp)$ since $\Stab(\KActp) = \bigcap_{\text{branches}~\K'} \Stab(\K')$.
            Towards a contradiction, suppose there is some cardinal $\lambda \geq \cardS + \aleph_0$ such that $\K_0$ is stable in $\lambda$, but $\KActp$ is not.
            In this case, there is some $S$-act $M$ with $\card{M} = \lambda$ and $\card{\gS_{\KActp}(M)} > \lambda$.
            Choose pairwise distinct $\{p_i \coloneqq \gtp_{\KActp}(a_i/M;A_i) \mid i < \lambda^+\} \subseteq \gS_{\KActp}(M)$.
            Since $\KActp$ has amalgamation, we may assume that $A_i = A_0 \eqqcolon A$ for every $i < \lambda^+$.
            Without loss of generality, we may assume that $\Theta$ is disjoint from $M$ and $A$.
            Observe that $M \cup \Theta \in (\K_0)_\lambda$\footnote{
                This is the only place in the proof where the argument fails for a branch $\K' \neq \K_0$ because, in general, $B \in \K'$ does not imply that $M \coprod B \in \K'$.
                However, this is true when $B = \Theta$ and $\K' = \K_0$.
            }
            and $M \cup \Theta \leq_p A \cup \Theta$ by Fact~\ref{fact: preserving purity}(2).
            Thus, since $\K_0$ is stable in $\lambda$,
            \[
                \left|\left\lbrace\gtp_{\K_0}\left(a_i/\left(M \cup \Theta\right);A \cup \Theta\right) \mid i < \lambda^+\right\rbrace\right| \leq \lambda.
            \]
            In particular, there are distinct $i,j < \lambda^+$ such that
            \[\gtp_{\K_0}\left(a_i/\left(M \cup \Theta\right);A \cup \Theta \right) = \gtp_{\K_0}\left(a_j/\left(M \cup \Theta\right);A\cup\Theta\right).\]
            It follows from Theorem \ref{thm: gtp = pp} (see Remark \ref{rem: branch stability and tameness}) that $\pp\left(a_i/\left(M \cup \Theta\right);A\cup \Theta\right) = \pp\left(a_j/\left(M \cup \Theta\right);A \cup\Theta\right)$.
            We claim that $\pp(a_i/M;A) = \pp(a_j/M;A)$ which, together with Theorem \ref{thm: gtp = pp}, contradicts the choice of $p_i,p_j \in \gS_{\KActp}(M)$.

            Let $\Sigma$ be a finite system of equations in variables $x,y_1, \dots, y_n$ with constants from $M$ and solution $a_i,b_1, \dots, b_n \in A$.
            Since $\pp\left(a_i/\left(M \cup \Theta\right);A\cup\Theta\right) = \pp\left(a_j/\left(M \cup \Theta\right);A\cup\Theta\right)$, $\Sigma$ also has a solution $a_j,c_1, \dots, c_n \in A \cup \Theta$, where we relabel the variables as necessary so that $c_1, \dots, c_{k_0} \in A$ and $c_{k_0 + 1} = \ldots = c_n = \theta$ for some $0 \leq k_0 \leq n$.
            Note that $\Sigma$ does not contain an equation of the form $sy_k = ty_\ell$ for some $s,t \in S$ and $k \leq k_0 < \ell$ because otherwise $\theta = t\theta = tc_\ell = sc_k \in A$, a contradiction.
            Similarly, $\Sigma$ does not contain an equation of the form $sx = ty_\ell$ for $s,t \in S$ and $\ell > k_0$.
            Thus, we may partition $\Sigma$ into a system of equations $\Sigma_A$ in variables $x,y_1, \dots, y_{k_0}$, and a system of equations $\Sigma_\Theta$ in variables $y_{k_0 + 1}, \dots, y_n$.
            Since $a_i,b_1, \dots, b_n$ is a solution to $\Sigma$, we know that, in particular, $b_{k_0 + 1}, \dots, b_n \in A$ is a solution to $\Sigma_\Theta$.
            Similarly, $a_j, c_1, \dots, c_{k_0} \in A$ is a solution to $\Sigma_A$.
            Hence, $a_j,c_1, \dots, c_{k_0},b_{k_0 + 1}, \dots, b_n \in A$ is a solution to $\Sigma$, and so $\pp(a_i/M;A) \subseteq \pp(a_j/M;A)$.
            Showing that $\pp(a_i/M;A) \supseteq \pp(a_j/M;A)$ is analogous and omitted.
        \end{proof}

    The remainder of our argument makes use of the following notions from \cite[Definition 5.9]{beard2025spectrum}, see also \cite[Definitions 2.4 and 4.6]{vasey2018toward}.

    \begin{definition}
        Let $\nf$ be a weakly stable independence relation on an AEC $\K$ with amalgamation, joint embedding, and no maximal models. 
        \begin{enumerate}
            \item[$(1)$] For $\lambda \in \Stab(\K)$, we define $\underline{\kappa}(\nf,\K_\lambda,\leq_{\K}^u)$ to be the set of all regular cardinals $\delta < \lambda^+$ such that whenever $\{M_i \mid i \leq \delta\}$ is a $\leq_\K^u$-increasing continuous chain in $\K_\lambda$ and $p \in \gS(M_\delta)$, then there is some $i < \delta$ such that $p$ does not $\nf$-fork over $M_i$. 
            If $\underline{\kappa}(\nf,\textbf{K}_\lambda,\leq_{\textbf{K}}^u) \neq \varnothing$, let $\kappa(\nf,\K_\lambda,\leq_{\K}^u) = \min\underline{\kappa}(\nf,\textbf{K}_\lambda,\leq_{\textbf{K}}^u)$.
            
            \item[$(2)$] Define
            \[
                \underline{\chi}(\nf,\K,\leq_\K^u) = \bigcup_{\lambda \in \Stab(\K)}\underline{\kappa}(\nf,\K_\lambda,\leq_\K^u).
            \]
            If $\underline{\chi}(\nf,\K,\leq_\K^u) \neq \varnothing$, let $\chi(\nf,\K,\leq_\K^u) = \min\underline{\chi}(\nf,\K,\leq_\K^u)$.
        \end{enumerate}
    \end{definition}

    For $\lambda \in \Stab(\K_0)$, we will simply write $\underline{\kappa}(\lambda)$, $\kappa(\lambda)$, $\underline{\chi}$, and $\chi$ to denote $\underline{\kappa}(\nfp\upharpoonright \K_0,(\K_0)_\lambda,\leq_p^u)$, $\kappa(\nfp\upharpoonright\K_0,(\K_0)_\lambda,\leq_p^u)$, $\underline{\chi}(\nfp\upharpoonright\K_0,\K_0,\leq_p^u)$, and $\chi(\nfp\upharpoonright\K_0,\K_0,\leq_p^u)$, respectively.

    The following result is an immediate consequence of Lemma \ref{lem: gamma in underline{kappa}} and the above definitions. 

    \begin{corollary}\label{cor: upper bound for kappa}
        Let $S$ be an LO monoid and $\lambda \in \Stab(\K_0)$ with $\lambda > \cardS+\aleph_0$.
        Then $\gamma_r(S) \in \underline{\kappa}(\lambda)$ and, in particular, $\chi \leq \gamma_r(S)$.
    \end{corollary}

    Our next objective is to show that the reverse inequality also holds, i.e., $\gamma_r(S) \leq \chi$.
    We use the fact below to verify that $\K_0$ satisfies all of the hypotheses used in \cite[Section 5]{beard2025spectrum}, which will be of relevance throughout the end of the section.

    \begin{fact}[{\cite[Lemma 6.1]{beard2025spectrum}\label{fact: checking hypotheses}}]
        Suppose $\K$ is an AEC with joint embedding, no maximal models, $(<\aleph_0)$-tameness, and a weakly stable independence relation with strong $\LS(\K)$-local character.
        Then $\K$ satisfies \cite[Hypothesis 5.5]{beard2025spectrum}.
    \end{fact}

    \begin{remark}\label{rem: K0 satisfies hyp5.5}
        In particular, for any LO monoid $S$, $\K_0$ satisfies \cite[Hypothesis 5.5]{beard2025spectrum} since the hypotheses of Fact \ref{fact: checking hypotheses} were verified for $\K_0$ in Fact \ref{fact: branch AECs}, Remark \ref{rem: branch stability and tameness}, and Remark \ref{rem: branch ind. rel.}.
    \end{remark}

    \begin{remark}
        The definitions for $\underline{\kappa}(\nf,\K_\lambda,\leq_{\K}^u)$, $\kappa(\nf,\K_\lambda,\leq_{\K}^u)$, $\underline{\chi}(\nf,\K,\leq_\K^u)$, and $\chi(\nf,\K,\leq_\K^u)$ above differ from \cite{vasey2018toward} where splitting is used rather than an arbitrary independence relation.
        However, by Remark~\ref{rem: K0 satisfies hyp5.5}, we may apply \cite[Corollary 5.11]{beard2025spectrum} to see that the definitions agree in our setup. 
    \end{remark}

    The next result will use \cite[Corollary 11.4]{vasey2018toward}.
    The hypotheses there are for $\K$ to be an $\LS(\K)$-tame AEC with amalgamation, joint embedding, no maximal models, and for splitting to have weak continuity.
    It is clear that $\K_0$ satisfies the first four of these hypotheses by Remark \ref{rem: branch stability and tameness} and Fact \ref{fact: branch AECs}, while the last hypothesis follows from $(< \aleph_0)$-tameness and \cite[Lemma 3.2]{mazari2024existence}.\footnote{
        It is straightforward to show that $(<\aleph_0)$-tameness implies $(<\lambda^+,\lambda)$-locality for every infinite cardinal $\lambda$, and that weak continuity of splitting follows from continuity in every $\lambda \in \Stab(\K_0)$.
        For the interested reader, the relevant definitions are \cite[Definitions 2.8 and 3.1]{mazari2024existence} and \cite[Definitions 2.4, 2.6, 3.8, and 11.1]{vasey2018toward}.
    }

    \begin{lemma}\label{lem: chi = gamma_r(S)}
        Let $\lambda_0 = \beth_{\beth_{\left(2^{\cardS+\aleph_0}\right)^+}}$ and assume that $\theta(\lambda_0)$ exists.
        If $S$ is an LO monoid, then $\chi = \gamma_r(S)$.
    \end{lemma}

        \begin{proof}
            Since $\chi \leq \gamma_r(S)$ by Corollary \ref{cor: upper bound for kappa}, it suffices to show that $\chi \geq \gamma_r(S)$.
            Let $\lambda_1 = \theta(\lambda_0)$ and observe that if $\mu < \chi$, then
            \[
                \beth_\chi(\lambda_1)^\mu = \left(\sup_{\alpha < \chi} \beth_\alpha(\lambda_1)\right)^\mu 
                \leq \sup_{\alpha < \chi} \beth_\alpha (\lambda_1)^\mu
                \leq \sup_{\alpha < \chi} \beth_{\alpha+1}(\lambda_1)^\mu 
                = \sup_{\alpha < \chi}\left(2^{\beth_\alpha(\lambda_1)}\right)^\mu
            \]
            \[
                = \sup_{\alpha < \chi} 2^{\beth_\alpha(\lambda_1)}
                = \sup_{\alpha < \chi}\beth_{\alpha + 1}(\lambda_1)
                = \sup_{\alpha < \chi}\beth_\alpha(\lambda_1) = \beth_\chi(\lambda_1),
            \]
            where $\left(\sup_{\alpha < \chi} \beth_\alpha(\lambda_1)\right)^\mu \leq \sup_{\alpha < \chi} \beth_\alpha (\lambda_1)^\mu$  because $\operatorname{cf}(\chi) = \chi > \mu$, and $\left(2^{\beth_\alpha(\lambda_1)}\right)^\mu = 2^{\beth_\alpha(\lambda_1)}$ for each $\alpha < \chi$ because $\mu < \chi \leq \gamma_r(S) \leq (\cardS+\aleph_0)^+ \leq \lambda_0 \leq \lambda_1 \leq \beth_\alpha(\lambda_1)$.
            It follows that
            \[
                \beth_\chi(\lambda_1)^{<\chi} = \sum_{\mu < \chi}\beth_{\chi}(\lambda_1)^\mu = \sum_{\mu < \chi}\beth_\chi(\lambda_1) = \beth_\chi(\lambda_1).
            \]
            By design, $\lambda_1$ was chosen large enough\footnote{
                This follows from \cite[Theorem 11.3 and Notation 2.1]{vasey2018toward} where, in our setting, $H_1 = \beth_{(2^{\cardS+\aleph_0})^+}$.
            } so that \cite[Corollary 11.4]{vasey2018toward} may be applied to see that $\K_0$ is stable in $\beth_\chi(\lambda_1)$.
            Therefore, Lemma \ref{lem: stability spectrum of zero branch}, Corollary \ref{cor: stab in pure -> lambda^(<gamma) = lambda}, and the regularity of $\chi$ imply that $\chi = \operatorname{cf}(\chi) = \operatorname{cf}(\beth_\chi(\lambda_1)) \geq \gamma_r(S)$.
        \end{proof}

    Assuming a set-theoretic hypothesis similar to the one used in Theorem \ref{thm: GCH stability}, we can apply \cite[Theorem~5.24]{beard2025spectrum} to characterize the spectrum of limit models of $\K_0$.

    \begin{theorem}\label{thm: limit model iso}
        Let $\lambda_0 = \beth_{\beth_{\left(2^{\cardS+\aleph_0}\right)^+}}$ and assume that $\theta(\lambda_0)$ exists.
        If $\K_0$ is stable in $\lambda \geq \lambda_0$, and $A_1, A_2, M \in \K_0$ with $A_i$ a $(\lambda,\delta_i)$-limit model (over $M$) for $i = 1,2$ and $\delta_1,\delta_2 < \lambda^+$, then $A_1$ is isomorphic to $A_2$ (over $M$) if and only if $\operatorname{cf}(\delta_1),\operatorname{cf}(\delta_2) \geq \gamma_r(S)$.
    \end{theorem}

        \begin{proof}
            We know by Lemma \ref{lem: chi = gamma_r(S)} that $\chi = \gamma_r(S)$ and so, the result follows from Remark \ref{rem: K0 satisfies hyp5.5} and \cite[Theorem~5.24]{beard2025spectrum}.
        \end{proof}

    As with Theorem \ref{thm: GCH stability}, we believe that Theorem \ref{thm: limit model iso} might actually be obtainable in ZFC, but do not know how to show it.

    We also do not know if Theorem \ref{thm: limit model iso} holds for other branches of $\KActp$, but it would hold for any branch such that the stability spectrum of the branch is the same as the stability spectrum of $\KActp$ above some cardinal as this is the only special property of $\K_0$ we used when we calculated $\chi$ in Lemma \ref{lem: chi = gamma_r(S)}. In particular, we have:
     
    \begin{lemma}
        Let $\lambda_0 = \beth_{\beth_{\left(2^{\cardS+\aleph_0}\right)^+}}$ and assume that $\theta(\lambda_0)$ exists.
        If $\KActp$ is superstable, $\lambda \geq \lambda_0$, and $A_1, A_2, M \in \KActp$ with $A_i$ a $(\lambda,\delta_i)$-limit model over $M$ for $i = 1,2$ and $\delta_1,\delta_2 < \lambda^+$, then $A_1$ is isomorphic to $A_2$ over $M$.
    \end{lemma}
     
        \begin{proof} 
            Since $A_1, A_2$ are limit models over $M$, we have that $A_1, A_2$ are in the same branch $\K'$ of  $\KActp$. Because $\KActp$ is superstable, it follows from the proof of Theorem \ref{thm: WO <-> Superstable} that  both $\KActp$ and $\K'$ are stable in every cardinal $\lambda \geq 2^{\cardS + \aleph_0}$. Arguing as for $\K_0$, we have that $\K'$ satisfies \cite[Hypothesis 5.5]{beard2025spectrum} (see Remark \ref{rem: K0 satisfies hyp5.5}) and $ \chi(\nfp\upharpoonright\K',\K',\leq_p^u)=\gamma_r(S)$ (see Lemma \ref{lem: chi = gamma_r(S)}). Since $S$ is weakly noetherian by Theorem~\ref{thm: WO <-> Superstable}, we have that $\gamma_r(S) = \aleph_0$. The result then follows from \cite[Theorem~5.24]{beard2025spectrum}.
        \end{proof}

    \begin{remark}
    If $\K$ is an $\LS(\K)$-tame AEC with joint embedding, amalgamation, and no maximal models, then $\K$ is superstable if and only if there is some $\mu \geq \LS(\K)$ such that $\K$ has a unique (up to isomorphism) $\lambda$-limit model for every $\lambda \geq \mu$ by \cite[Corollary 1.3]{grossberg2017equivalent} and \cite[Corollary 4.24]{vasey2018toward}.
        If $S$ is LO and weakly noetherian but does not have local zeros, then $\KActp$ is superstable but does not have joint embedding. 
        Hence $\KActp$ does not have 
        % {\color{red}``globally"} 
        unique $\lambda$-limit models; however the previous result shows that we have unique $\lambda$-limit models over the same basis.
    \end{remark}
   
\bibliographystyle{abbrvnat}
\bibliography{refs}
\end{document}